\documentclass[12pt, parskip=full]{scrartcl}
\pdfoutput=1
\usepackage[ngerman, english]{babel}
\usepackage[utf8]{inputenc}
\setkomafont{sectioning}{\normalfont\normalcolor\bfseries}

\usepackage{amsmath}
\usepackage{amsfonts}
\usepackage{amssymb}
\usepackage{amsthm}
\usepackage{stmaryrd}

\theoremstyle{plain}
\newtheorem{theorem}{Theorem}[section]
\newtheorem{lemma}[theorem]{Lemma}
\newtheorem{corollary}[theorem]{Corollary}
\newtheorem{assumption}{Assumption}

\theoremstyle{definition}
\newtheorem{definition}{Definition}[section]

\theoremstyle{remark}

\newtheorem{remark}[theorem]{Remark}

\usepackage[round]{natbib}
\usepackage{url}
\usepackage{abstract}

\usepackage{hyperref}
\usepackage{graphicx}
\usepackage{xcolor}

\usepackage{pgf}
\usepackage{tikz}
\usetikzlibrary{arrows,automata,positioning,calc}

\title{Stationary Equilibria of Mean Field Games with Finite State and Action Space}
\author{Berenice Anne Neumann\thanks{Universit\"at Hamburg, Department of Mathematics, STSP, Bundesstr. 55 (Geomatikum), 20146 Hamburg, Germany, E-mail address: berenice.neumann@uni-hamburg.de}}

\begin{document}
	
	\maketitle
	
	\begin{abstract}
		Mean field games formalize dynamic games with a continuum of players and explicit interaction where the players can have heterogeneous states. As they additionally yield approximate equilibria of corresponding $N$-player games, they are of great interest for socio-economic applications. 
		However, most techniques used for mean field games rely on assumptions that imply that for each population distribution there is a unique optimizer of the Hamiltonian. 
		For finite action spaces, this will only hold for trivial models. 
		Thus, the techniques used so far are not applicable. 
		We propose a model with finite state and action space, where the dynamics are given by a time-inhomogeneous Markov chain that might depend on the current population distribution. 
		We show existence of stationary mean field equilibria  {in mixed strategies} under mild assumptions and propose techniques to compute all these equilibria. 
		More precisely, our results allow –- given that the generators are irreducible -- to characterize the set of stationary mean field equilibria as the set of all fixed points of a map completely characterized by the transition rates and rewards for deterministic strategies. 
		Additionally, we propose several partial results for the case of non-irreducible generators and we demonstrate the presented techniques on two examples. 
	\end{abstract}
	
	\section{Introduction}
	
	Mean field games have been introduced independently by \citet{LasryJapanese2007} and \citet{HuangNCE2006} in order to provide a framework for dynamic stochastic games in continuous time with a continuum of players, whose equilibria furthermore serve as approximate Nash equilibria for corresponding $N$-player games. The main feature of these games is that any player does not observe the state and action of each other player individually, but only at an aggregated level through the empirical distribution of these characteristics. 
	
	From an economic perspective these games are of particular interest as they yield a formal way to describe games with a continuum of rational players that accounts for explicit interaction (in contrast to the classical assumption in general equilibrium theory that ``prices mediate all social interaction'' \citep{ParisPrinceton2010}) as well as heterogeneity of states (in contrast to representative agent models (see \citet[p.209]{GomesEconomicModelsMFG2015})). Therefore a wide range of economic models relying on mean field games emerged, which includes growth models, the production of an exhaustible resource by a continuum of producers as well as opinion dynamics (see \citet{GomesEconomicModelsMFG2015}, \citet{ParisPrinceton2010}, \citet{CainesHandbook} and the references therein).
	
	In classical mean field games models each individual player's dynamics is given by a diffusion process whose drift and volatility depend on time, the current state and the current action of the individual player as well as the current distribution of all players.
	Each player individually solves an optimal control problem given these dynamics, where the costs also depend on the current state and action of the individual player as well as distribution of all players.
	A mean field equilibrium is then given by a flow of population distributions $m$ such that there is an optimal strategy $\pi$ for the individual control problem given $m$ and the distribution of the individual player given this strategy is in turn $m$. 
	This is a natural analogue of a Nash equilibrium of games with $N$ players: Given that all players play the strategy $\pi$, no player wants to deviate from playing $\pi$, as the population distribution is $m$ and $\pi$ is a best response to it.
	
	To solve this type of mean field games one then sets up several assumptions, which usually include the assumption that there is a unique optimizer of the Hamiltonian.
	Then one can show that finding a mean field equilibrium boils down to solving a system of a Hamilton-Jacobi-Bellman equation coupled with a Fokker-Planck-equation -- or if one prefers the probabilistic approach -- to a Forward-Backward Stochastic Differential Equation (FBSDE). 
	The forward-backward structure of these differential equations is non-standard and a wide range of the literature regarding mean field games covers the analysis of these equations. 
	For more details consider \citet{BensoussanMFG} and \citet{CarmonaMFG2018, CarmonaMFGPartTwo2018} as well as the references therein.
	{We remark, that in \citet{LackerMixedStrategy} existence of mean field equilibria in mixed strategies is proven under continuity, measureability and boundedness conditions, which, in particular, do not include the assumption that a unique optimizer of the Hamiltonian exists.
		A similar existence result for mean field games with controlled jump-diffusion dynamics can be found in \citet{BennazzoliMixedStrategy}. }
	
	{In \citet{GomesDiscrete2010,GomesConti2013} mean field games with finite state space have been introduced and thereafter several other (more general) mean field game models with finite state spaces haven been considered. 
		We give an overview at the end of the introduction.
		Also applications of mean field games with finite state space have been considered.
		However, several applications consider finite action spaces (for example \citet{KolokoltsovBotnet2016}, \citet{KolokoltsovCorruption2017}, \citet{GueantDiss2009} and \citet{BesancenotParadigm2015}) and for this type of models the literature covers only an existence result in mixed strategies (also called relaxed strategies) (see \citet{CecchinProbabilistic2018}). Indeed, for non-trivial models with finite action spaces there are always population distributions for which more than one optimizer of the Hamiltonian exists, in which case most of the techniques presented in the literature so far are not applicable. 
		For this reason for the previously mentioned examples the} authors develop their own tools to solve their particular model: Kolokoltsov and Bensoussan(2016) and Kolokoltsov and Malafeyev(2017) only analyse stationary equilibria in deterministic strategies, Gueant(2009a) and Besancenot and Dogguy(2015) set up a dynamics equation only after analysing optimal decisions given a certain population distribution. 
	They then also focus on stationary equilibria, as well as the dynamic behaviour close to these stationary equilibria and the effect of shocks.
	
	In this paper, we will present general tools to compute stationary equilibria of mean field games with finite state and action space: We introduce the notion of stationary mean field equilibria into the mean field game model with finite state and finite action space presented in \citet{DoncelExplicit2016}. 
	{We remark that we consider a stationary equilibrium in an infinite horizon mean field game where the expected discounted reward ($\int_0^\infty e^{-\beta t} \ldots$) is maximized and not, as in many other settings, a stationary (ergodic) equilibrium, where the expected average reward ($\limsup_{T \rightarrow \infty} \int_0^T \ldots$) is maximized.}
	Since the analytic formulation of \citet{DoncelExplicit2016} is not suitable for this task, we formulate the model in a probabilistic way: 
	The individual dynamics of a player is given by a time-inhomogeneous continuous time Markov chain with the generator depending on the current action and the current distribution of the population; the costs depend on the current state as well as the current action of the individual player and the current population distribution. 
	As in \citet{DoncelExplicit2016} in the case of dynamic equilibria, we only prove existence of stationary mean field equilibria in mixed strategies.

	We focus on stationary equilibria for several reasons: First, searching for stationary equilibria reduces the complexity of the considered problem. More precisely, we can utilize the standard theory on Markov decision process with stationary transition rates and rewards and we can considered a fixed point problem in $\mathbb{R}^S$ instead of a fixed point problem in some function space. 
	Second, the main focus of the economic models studied so far also lies in stationary equilibria. 
	Third and linked to the last reason, one can often establish some kind of convergence/adjustment process towards stationary equilibria. 
	More precisely, \citet{GomesConti2013} prove that  {for mean field games with finite state space} under certain assumptions every dynamic mean field equilibrium converges to the stationary  {(ergodic)} equilibrium.
	{In the case of continuous state spaces \citet{CardaliaguetContiTrend,CardaliaguetContiTrend2} proves that under certain assumptions every dynamic mean field equilibrium converges to the stationary (ergodic) equilibrium} and
	\citet{GueantReference2009} describes a cognitive process that converges if it is started close to stationary equilibrium indeed to the stationary  {(ergodic)} equilibrium. 
	Moreover, in the examples presented by \citet{GueantDiss2009} and \citet{BesancenotParadigm2015} it is shown that there is local convergence of the trajectories of the dynamic equilibrium to the stationary  {(discounted reward)} equilibrium. 
	
	Relying on our probabilistic formulation of the model, we will show existence of stationary equilibria under the same conditions as in the dynamic case. This is compared to \citet{GomesConti2013} a surprising result as they need several additional assumptions compared to the dynamic existence result to establish existence of stationary equilibria.
	
	Thereafter, we derive tools to compute all stationary equilibria (including those where the equilibrium strategy randomizes over different actions). As in standard mean field game models, we first have to solve an optimal control problem given a fixed flow of population distribution and second a fixed point problem, namely searching for flows of population distributions $m$ such that $m$ is the distribution of an individual player playing optimal given $m$.
	We will see that the first problem is equivalent to solving a standard Markov decision process with expected discounted reward criterion and we will show that the set of all randomized optimal stationary strategies is the convex hull of all deterministic optimal stationary strategies. 
	For general dynamics, we cannot simplify the fixed point problem directly but we will provide a generally applicable reformulation of the necessary and sufficient balance equations inspired by the cut criterion for standard Markov chains, which often proves helpful in examples.
	Assuming irreducibility of the generators of the individual dynamics given any population distribution and any strategy, we can obtain all distributions of stationary equilibria as the fixed points of a set-valued map with convex values, which can be completely characterized by the transition rates and rewards given deterministic strategies.
	
	\subsection{Related Literature}
	
	As indicated previously we would like to sketch briefly the research regarding finite state mean field games:
	The starting point regarding the study of finite state mean field games were the models of \citet{GomesDiscrete2010} (in discrete time), \citet{GomesConti2013} and \citet{GueantReduction2011, GueantCongestion2015}.
	We will focus on the continuous time models here: In both models, fully controllable transition rates are considered (\citet{GueantReduction2011, GueantCongestion2015} additionally assumes that the players might not reach all other states from a given state)  such that the player's dynamics is given by a time-inhomogeneous Markov chain. 
	The costs consist of instantaneous costs depending on the current state and action of the individual player as well as the current distribution of all players together with a terminal cost depending on the current state of the individual player and the current distribution of all players. 
	In both models, assumptions are set up such that there is always a unique optimizer of the Hamiltonian.
	For both models then a system of forward-backward ODEs is presented, the solution of which yields a mean field equilibrium. 
	Moreover, existence and uniqueness of solutions to these equations is discussed.
	
	\citet{GueantReduction2011} discusses further sufficient conditions for the existence of mean field equilibria (including a discrete state master equation); \citet{GomesConti2013} studies stationary equilibria and establishes for contractive mean field games that a trend to equilibrium exists. 
	\citet{GomesConti2013} furthermore study an $N$-player game and the convergence of this game to the mean field game model. 
	Additionally, as in the diffusion-based models, the class of potential mean field games is introduced, which has a simpler cost structure and, thus, allows for deeper results. Namely, the costs split in two additive terms, one term depending on the current state and the current population distribution,  {which is furthermore the gradient of a convex function}, as well as one term depending on the current state and current action. 
	
	Several other authors discuss similar questions in models with more general individual dynamics, in particular the transition rates might not be fully controllable, but again assumptions were set which imply that there is a unique optimizer of the Hamiltonian:
	\citet[Chapter 7.2]{CarmonaMFG2018} provide a discussion of finite state mean field games, which is closely related to their exposition of standard mean field games models with continuous state space. They consider models in which transition rates depend on the current state, the current action and the population distribution and discuss existence and uniqueness results as well as a master equation. 
	\citet{BasnaConvergence2014} discusses mean field games were the dynamics are given by a non-linear Markov process with a generator that might additionally depend on the distribution of all other players and show under several assumptions that mean field equilibria are $\frac{1}{N}$-Nash equilibria for the corresponding $N$-player games. 
	
	\citet{CecchinProbabilistic2018} present a mean field game with the individual dynamics given by stochastic differential equation driven by a stationary Poisson random measure, where again a dependence on the current population distribution is possible. 
	They discuss existence  {(also in mixed strategies under mild continuity and boundedness assumptions)} and uniqueness of mean field equilibria and furthermore show that mean field equilibria in open-loop and feedback strategies are $\epsilon_N$-Nash equilibria for the corresponding $N$-player game. 
	\citet{CarmonaFinite} provide a mean field game model with the dynamics given by a continuous time Markov chain with a generator which might depend on the current distributions of the states as well as the actions of all players. 
	Using the semimartingale representation of continuous time Markov chains they again consider existence, uniqueness and the question when a mean field equilibrium is an approximate Nash equilibrium for the corresponding $N$-player game.
	
	The model of \citet{DoncelExplicit2016} (which we consider in this paper) does not require a unique optimizer of the Hamiltonian, but instead it is directly assumed that the action space is finite. 
	The dynamics of each individual player are given by a differential equation specifying the transition rates, which in turn depend on the individual's state and action as well as on the current population distribution. 
	Existence of dynamic mean field equilibria is shown and a discrete time $N$-player game is considered, for which it is shown that mean field equilibria are $\epsilon$-Nash equilibria.  
	Furthermore, the question of convergence is considered.  More precisely, given a sequence of strategies $(\pi^N)_{N \in \mathbb{N}}$ which are equilibria in the $N$-player game is there some sub-sequence converging to a mean field equilibrium? 
	The answer to this question is positive if one considers local strategies (which only depend on the current state and time), but negative if one considers Markov strategies (which also depend on the current distribution of all players). 
	The intuitive reason for this is that the ``tit-for-tat''-principle cannot be applied in the limit (see \citet{DoncelShort} for more details). 
	
	\subsection{Organization of the Paper}
	
	The structure of the paper is as follows: Section \ref{Section:TheModel} introduces the model in a probabilistic formulation.
	{Section \ref{Section:IndivControl} discusses the individual control problem.} 
	In Section \ref{Section:Existence} we show that for all models fitting into our framework a stationary mean field equilibrium in mixed strategies exists.
	Section \ref{Section:FixedPointProblem} first discusses the generally applicable cut criterion for our setting, then - given irreducibility of the generator - we propose a characterization of the set of all distributions of stationary mean field equilibria as the set of all fixed points of a suitable set-valued map which is completely determined by the dynamics and rewards given the deterministic strategies.
	Section \ref{Section:Examples} concludes the paper by showing how the presented tools can be applied to find stationary mean field equilibria in two examples.

	\section{The Model}
	\label{Section:TheModel}
	
	Let $\mathcal{S}=\{1, \ldots, S\}$ ($S>1$) be the set of possible states of each player and let $\mathcal{A}=\{1, \ldots, A\}$ be the set of possible actions. 
	With $\mathcal{P}(\mathcal{S})$ we denote the probability simplex over $\mathcal{S}$ and similarly for $\mathcal{P}(\mathcal{A})$. 
	A (mixed) strategy is a measurable function $\pi: \mathcal{S} \times [0,\infty) \rightarrow \mathcal{P}(\mathcal{A})$, $(i,t) \mapsto (\pi_{ia}(t))_{a \in \mathcal{A}}$ with the interpretation that $\pi_{ia}(t)$ is the probability that at time $t$ and in state $i$ the player chooses action $a$. 
	We say that a strategy $\pi=d:\mathcal{S} \times [0,\infty) \rightarrow \mathcal{P}(\mathcal{A})$ is \textit{deterministic} if it satisfies for all $t \ge 0$ and for all $i \in \mathcal{S}$ that there is an $a \in \mathcal{A}$ such that $d_{ia}(t)=1$ and $d_{ia'}=0$ for all $a' \in \mathcal{A} \setminus \{a\}$. 
	Throughout the presentation we often use the following equivalent representation, which is to represent a deterministic strategy as a function $d: \mathcal{S} \times [0,\infty) \rightarrow \mathcal{A}, (i,t) \mapsto d_i(t)$ with the interpretation that $d_i(t)=a$ states that at time $t$ in state $i$ action $a$ is chosen. 
	With $\Pi$ we denote the set of all (mixed) strategies and with $D$ the set of all deterministic strategies.
	
	The individual dynamics of each player given a flow of population distributions $m: [0, \infty) \rightarrow \mathcal{P}(\mathcal{S})$ and a strategy $\pi: \mathcal{S} \times [0, \infty) \rightarrow \mathcal{P}(\mathcal{A})$ are given as a Markov process $X^\pi(m)$  {on a given probability space $(\Omega, \mathcal{F}, \mathbb{P})$} with given initial distribution $x_0$ and infinitesimal generator given by the $Q(t)$-matrix $$\left( Q^\pi(m(t),t) \right)_{ij} = \sum_{a \in \mathcal{A}} Q_{ija}(m(t)) \pi_{ia}(t),$$  where for all $a \in \mathcal{A}$ and $m \in \mathcal{P}(\mathcal{S})$ the matrices $(Q_{\cdot \cdot a}(m))$ are conservative generators, that is $Q_{ija}(m) \ge 0$ for all $i, j \in \mathcal{S}$ with $i \neq j$ and $\sum_{j \in \mathcal{S}} Q_{ija}(m) =0$ for all $i \in \mathcal{S}$. 
	
	{Given the initial condition $x_0 \in \mathcal{P}(\mathcal{S})$,} the goal of each player is to maximize his expected reward, which is given by 
	\begin{equation}
	\label{valueFunction}
	V {_{x_0}}(\pi^0, m) = \int_0^\infty \left( \sum_{i \in \mathcal{S}} \sum_{a \in \mathcal{A}} x_i^{\pi_0} (t) r_{ia} (m(t)) \pi_{i,a}^0(t) \right) e^{-\beta t} \text{d}t,
	\end{equation} where $x_i(t)$ is the probability that the individual player is in state $i$ at time $t$, $r: \mathcal{S} \times \mathcal{A} \times \mathcal{P}(\mathcal{S}) \rightarrow \mathbb{R}$ is a real-valued function and $\beta \in (0,1)$ is the discount factor. 
	That is, for a fixed population distribution $m: [0,\infty) \mapsto \mathcal{P}(\mathcal{S})$ we face a Markov decision process with expected discounted reward criterion and time-inhomogeneous reward functions and transition rates.
	
	We will work under the following mild continuity assumption, which will ensure that there is indeed a Markov process with generator $Q^\pi$ (see \citet[Appendix B+C]{GuoCTMDP2009} for details):
	
	\begin{assumption}
		\label{assumption:continuous}
		For all $i, j \in \mathcal{S}$ and all $a \in \mathcal{A}$ the function $m \mapsto Q_{ija}(m)$ mapping from $\mathcal{P}(\mathcal{S})$ to $\mathbb{R}$ is Lipschitz-continuous in $m$ . 
		For all $i \in \mathcal{S}$ and all $a \in \mathcal{A}$ the function $m \mapsto r_{ia}(m)$ mapping from $\mathcal{P}(\mathcal{S})$ to $\mathbb{R}$ is continuous in $m$.
	\end{assumption}
	
	With these preparations, we can introduce the concept of dynamic mean field equilibria:
	
	\begin{definition}
		\label{definition:Dynamic}
		Given an initial distribution $m_0$, a \textit{mean field equilibrium} is a flow of population distributions $m: [0,\infty) \rightarrow \mathcal{P}(\mathcal{S})$ with $m(0)=m_0$ and a strategy $\pi: \mathcal{S} \times [0,\infty) \rightarrow \mathcal{P}(\mathcal{A})$ such that
		\begin{itemize}
			\item the distribution of the process $X^\pi(m)$ at time $t$ is given by $m(t)$
			\item $V {_{m_0}}(\pi, m) \ge V {_{m_0}}(\pi',m)$ for all $\pi' \in \Pi$		. 
		\end{itemize}
	\end{definition} 
	As in standard game theory, our concept of mean field equilibrium captures the intuitive idea that no player wants to deviate: Given that all players play according to strategy $\pi$ the population's distribution will be $m$. If an individual player now evaluates whether he wants to deviate from playing $\pi$ he asks whether there is a strategy that yields a higher payoff given $m$. 
	Due to the second condition this is not possible. 
	Therefore, we indeed face an equilibrium in the standard economic sense.
	
	\begin{remark}
		Using the Kolmogorov forward equation \citep[Proposition C.4]{GuoCTMDP2009}, we see that the first condition implies the analytic condition used in \citet{DoncelExplicit2016} to characterize mean field equilibria, which states that $m$ is solution of  $$\dot{m}_j(t) = \sum_{i \in \mathcal{S}} m_i(t) Q_{ij}^\pi(m(t),t) \quad \forall j \in \mathcal{S} $$  with initial condition $m(0)=m^0$.
	\end{remark}
	
	\begin{remark}
		{The definition of strategies we adopt here is unusual in classical game theory.
			However, in the setting of mean field games it is sensible.
			Indeed, given the initial state of the system and the strategy of the other players the behaviour of the system is fully determined. 
			Thus, it suffices for the individual agent to know the initial global state (see \citet{CainesHandbook}).}
	\end{remark}

	In order to define stationary mean field equilibria, we first introduce the notion of stationary strategies: A \textit{stationary strategy} is a map $\pi: \mathcal{S} \times [0,\infty) \rightarrow \mathcal{P}(\mathcal{A})$ such that $\pi_{ia}(t) = \pi_{ia}$ for all $t \ge 0$. 
	Again we denote by $\Pi^s$ the set of all stationary strategies and by $D^s$ the set of all deterministic stationary strategies.
	
	\begin{definition}
		\label{definition:Stationary}
		A stationary mean field equilibrium is given by a stationary strategy $\pi$ and a vector $m \in \mathcal{P}(\mathcal{S})$ such that
		\begin{itemize}
			\item the law of $X^\pi(m)$ at any point in time $t$ is given by $m$ 
			\item  {for any initial distribution $x_0 \in \mathcal{P}(\mathcal{S})$ we have} $V {_{x_0}}(\pi, m) \ge V {_{x_0}}(\pi', m)$ for all $\pi' \in \Pi$.
		\end{itemize}
	\end{definition}
	
	This notion is a sensible formalization of stationary equilibria: Given the strategy $\pi$ the population's distribution will be $m$ for all time points.
	{An individual agent at a given time point can be in any state, however, if he evaluates whether he wants to deviate from playing $\pi$, the second condition ensures that this is not beneficial for him. 
		Thus, he has no incentive to deviate from the equilibrium strategy $\pi$, which means that the population will indeed remain in the stationary equilibrium regime of playing $\pi$.}
	
	\begin{remark}
		{We remark that the matrix $Q_{ij}^\pi(m,t)$ does not depend on $t$ in this context, therefore, we write $Q^\pi_{ij}(m):= Q^\pi_{ij}(m,t)$.}
		Using this, we obtain that the first condition is equivalent to  $$0 = \sum_{i \in \mathcal{S}}  m_i Q^\pi_{ij}(m) \quad \forall j \in \mathcal{S}.$$
		Moreover, we remark that the second condition requires $\pi$ to  {be} optimal among all strategies, not only those that are stationary.
	\end{remark}
	
	\begin{remark}
		In contrast to the standard models where the assumptions usually imply that a unique optimal best response exists, the mean field equilibria we consider are not fully specified by the distribution, as it might happen that several actions are simultaneously optimal and induce the same distribution.
		However, the dynamic mean field equilibrium is fully specified by describing the equilibrium strategy, as one can show using standard techniques \citep[Theorem 10.XX]{WalterODE1998}  that there is at most one solution to the differential equation. 
		For the stationary mean field equilibrium, this again does not hold true, as it might happen that given a strategy there are multiple stationary distributions. 
		For this reason, we define mean field equilibria always as pairs of the equilibrium distribution and the equilibrium strategy.
	\end{remark}
	
	\begin{remark}
		{We remark that for non-trivial models (in the sense that there is not one action that maximizes the Hamiltonian for every population distribution) we always obtain population distributions at which several actions maximize the Hamiltonian:
			Since $Q(\cdot)$ and $r(\cdot)$ are continuous, also the Hamiltonian is continuous in $m$. 
			Therefore, if we fix the costate variables, the sets of population distributions in which a particular action is a maximizer of the Hamiltonian are closed. Since the action space is finite and the set of all population distribution vectors is connected, we obtain that if there is more than one action that maximizes the Hamiltonian for some population distribution, then the set of population distributions where several actions simultaneously maximize the Hamiltonian is non-empty. 
			This implies that for the case of finite action spaces the assumption that a unique maximizer of the Hamiltonian exists is violated in all interesting cases. 
			Thus, new methods for the analysis of these models are necessary.}
	\end{remark}
	
	\section{The Individual Control Problem}
	\label{Section:IndivControl}
	
	This section is devoted to the analysis of the individual control problem: We propose a simple approach to determine which strategies are optimal for a given population distribution
	{and show that optimal stationary strategies are convex combinations of particular deterministic stationary strategies.}
	
	We start by showing that given a stationary population distribution $m \in \mathcal{P}(\mathcal{S})$ the individual player's control problem is equivalent to a continuous time Markov decision process with expected discounted reward criterion (see \citet{GuoCTMDP2009} for a definition and general results).
	
	\begin{lemma}
		\label{preliminaries:CTMDPEquivalence}
		Let $m \in \mathcal{P}(\mathcal{S})$ be a population distribution.
		A Markovian randomized strategy $\pi$ is optimal in our model given $m$, i.e. achieves  {for all $x_0 \in \mathcal{P}(\mathcal{S})$} the maximum value of $V {_{x_0}}(\pi',m)$ among all strategies $\pi' \in \Pi$, if and only if it is discounted reward optimal for the continuous time Markov decision process with transition rates $Q_{ija}(m)$, rewards $r_{ia}(m)$ and discount factor $\beta$.
		{In particular, there is a stationary strategy $\pi \in \Pi^s$ that satisfies $V_{x_0}(\pi,m) \ge V_{x_0}(\pi',m)$ for all $x_0 \in \mathcal{P}(\mathcal{S})$ and all $\pi' \in \Pi$.}
	\end{lemma}

	\begin{proof}
		Assumption \ref{assumption:continuous} ensures that the value function is finite for every population distribution function and every individual strategy since $r_{ia}(\cdot)$ is, as a continuous function on a compact space, uniformly bounded. 
		Thus, we can rewrite the value function by using the representation $x_i(t) = \sum_{k \in \mathcal{S}} x_k^0 \cdot p^{\pi^0}(0,k,t,i)$:
		\allowdisplaybreaks
		\begin{align*}
		V {_{x_0}}(\pi^0,m) &= \int_0^\infty \sum_{i \in \mathcal{S}} \sum_{a \in \mathcal{A}} x_i(t) r_{i,a}(m) \pi_{i,a}^0(t) e^{-\beta t} \text{d}t \\
		&= \int_0^\infty \sum_{i \in \mathcal{S}} \sum_{a \in \mathcal{A}} \sum_{k \in \mathcal{S}} x_k^0 p^{\pi^0}(0,k,t,i) r_{i,a}(m) \pi_{i,a}^0(t) e^{-\beta t} \text{d}t \\
		&= \sum_{k \in \mathcal{S}} x_k^0 \int_0^\infty e^{-\beta t} \sum_{i \in \mathcal{S}} \sum_{a \in \mathcal{A}} r_{i,a}(m) \pi_{ia}^0(t) p^{\pi^0}(0,k,t,i) \text{d} t \\
		&= \sum_{ k \in \mathcal{S}} x_k^0 \int_0^\infty e^{-\beta t} \mathbb{E}_k^{\pi^0} \left[ r(x(t), \pi_t^0)\right] \text{d} t \\
		&= \sum_{k \in \mathcal{S}} x_k^0 V_k^{\pi^0}(m),
		\end{align*} where $V_k^{\pi^0}(m)$ is the expected discounted reward of a continuous time Markov decision process with expected discounted reward criterion with the above-mentioned rates and rewards when the initial state is $k$. Since a strategy $\pi$ is optimal for the continuous time Markov decision process if it maximizes all $V_k^\pi(m)$ simultaneously, we obtain the desired equivalence.
		{The last statement directly follows for the classical theory for Markov decision process, see for example \citet[Chapter 4]{GuoCTMDP2009}.}
	\end{proof}
	
	Now, we show that the set of all optimal stationary strategies is the convex hull of all deterministic optimal stationary strategies.
	
	\begin{theorem}
		\label{preliminaries:OptStrategiesConvexConti}
		Let $m \in \mathcal{P}(\mathcal{S})$.
		Write $$\mathcal{D}(m)=\{d: \mathcal{S} \rightarrow \mathcal{A}: d(i) \in O_i(m) \text{ for all } i \in \mathcal{S}\}$$ with $$O_i(m) = \text{argmax}_{a \in  {\mathcal{A}}}  \left\{ r_{ia}(m) + \sum_{j \in \mathcal{S}} Q_{ija}(m) V_j^\ast(m) \right\},$$  {where $V^\ast(m)$ is the value function of the associated Markov decision process.}
		Then $\mathcal{D}(m)$ is non-empty. Furthermore, a stationary strategy is optimal for our model given $m$ if and only if it is a convex combination of strategies from $\mathcal{D}(m)$.
	\end{theorem}
	
	\begin{proof}
		By the previous lemma we know that the individual's control problem is equivalent to the continuous time Markov decision process with expected discounted reward criterion with discount factor $\beta$, transition rates $Q_{ija}(m)$ and reward function $r_{ia}(m)$.
		Since we consider a finite state space, we obtain, using the uniformization procedure (\citet[Remark 6.1]{GuoCTMDP2009}, \citet{KakumanuUniformization1977}), an equivalent discrete time Markov decision process with expected discounted reward criterion. 
		Its discount factor is $\alpha = \frac{||Q(m)||}{\beta + ||Q(m)||}$.  {Writing} $||Q(m)|| =  \sup_{i \in \mathcal{S}, a \in \mathcal{A}} -Q_{iia}(m)$, the transition rates are given by  $$\bar{P}_{ija}(m) = \frac{Q_{ija}(m)}{||Q(m)||} + \delta_{ij}$$ and the reward functions are given by $$\bar{r}_{ia}(m) = \frac{r_{ia}(m)}{\beta + ||Q(m)||}.$$
		Simple computations yield that $$O_i(m) = \text{argmax}_{a \in \mathcal{A}} \left\{ \bar{r}_{ia}(m) + \alpha \bar{P}_{ija}(m) \tilde{V}^\ast_j(m) \right\}$$ with $\tilde{V}^\ast_j(m)$ being the value function of the discrete time Markov decision process, which is indeed equal to the value function of the continuous time Markov decision process (see \citet{KakumanuUniformization1977} for details; note however, that he does not adjust the rewards, which yields the proportional factor for the value functions in his setting).
		
		Now we can prove the statement for discrete time Markov decision processes relying on the rich theory developed for these problems (see \citet{PutermanMDP1994}):	
		We first note that the set $\mathcal{D}(m)$ in non-empty since $\mathcal{A}$ is finite.
		Moreover, by \citet[Corollary 6.2.8]{PutermanMDP1994}, any deterministic strategy in our set $\mathcal{D}(m)$ is indeed optimal.
		Enumerate $\mathcal{D}(m) = \{d_1, \ldots, d_n\}$ and let $\pi \in \Pi^s$ be a convex combination of strategies in $\mathcal{D}(m)$, that is $$\pi = \sum_{i=1}^n \lambda_i d_i \quad \text{with} \quad \lambda_i \ge 0 \quad \forall i \in \{1, \ldots, n\} \quad \text{and} \quad \sum_{i=1}^n \lambda_i = 1.$$
		By \citet[Theorem 6.1.1]{PutermanMDP1994} the reward function given a certain stationary strategy can be written as the unique solution of $v= \bar{r}^\pi(m) + \alpha \bar{P}^\pi(m) v$ with $$\bar{r}^\pi(m)_i = \sum_{a \in \mathcal{A}} \pi_{ia} \bar{r}_{ia}(m) \text{ and } (\bar{P}^\pi)_{ij} = \sum_{a \in \mathcal{A}} \pi_{ia} \bar{P}_{ija}(m).$$ 
		Noting that $\bar{r}^\pi$ and $\bar{P}^\pi$ are linear in $\pi$, we can rewrite the policy evaluation equation as  
		$$V^\pi = \bar{r}^\pi + \alpha \bar{P}^\pi V^\pi = \sum_{i=1}^n \lambda_i (\bar{r}^{d_i} + \alpha \bar{P}^{d_i} V ^\pi).$$ 
		Since for all $i \in \{1, \ldots, n\}$ the deterministic stationary strategy $d_i$ is optimal it follows that $V^\pi = V^\ast$ is the unique solution of the policy evaluation equation: $$r^\pi + \alpha P^\pi V^\ast = \sum_{i=1}^n \lambda_i (r^{d_i} + \alpha P^{d_i} V^\ast) = \sum_{i=1}^n \lambda_i V^\ast = V^\ast.$$ 
		By \citet[Theorem 6.2.2 and Theorem 6.2.5]{PutermanMDP1994}, which states that in our setting the unique solution of the optimality equation is $V^\ast$, and by \citet[Theorem 6.2.6]{PutermanMDP1994}, which states that a strategy is optimal if and only if its value function is a solution of the optimality equation, we obtain that the strategy $\pi$ is optimal.
		
		To show the converse implication we assume that $\pi$ is not a convex combination of deterministic strategies from $\mathcal{D}(m)$. One easily sees that there is still a representation of $\pi$ as a convex combination of arbitrary deterministic strategies by considering 
		$$\pi = \sum_{d \in D^s} \lambda_d d \quad \text{with} \quad \lambda_d = \pi_{1d(1)} \cdot \ldots  {\cdot } \pi_{Sd(S)}.$$	
		Moreover, any convex combination of deterministic strategies representing the strategy $\pi$ has a summand $d \notin \mathcal{D}(m)$ with positive weight. 
		This means that for the strategy $d$ there is a state $i \in \mathcal{S}$ and an action $\tilde{a} \in \mathcal{A} \setminus O_i(m)$ such that $d(i)=\tilde{a}$. This implies that also the stationary strategy $\pi$ chooses that action $\tilde{a}$ in state $i$ with positive probability, that is $\pi_{i \tilde{a}} >0$. This means for the $i$-th component of the strategy's expected discounted reward:
		\allowdisplaybreaks
		\begin{align*}
		V^\pi(i) &= \bar{r}^\pi(i) + \sum_{j \in S} \alpha (\bar{P}^\pi)_{ij} V^\pi (j) \\
		&= \sum_{a \in  {\mathcal{A}}} \pi_{ia} \left( \bar{r}_{ia}(m) + \sum_{j \in S} \alpha \bar{P}_{ija}(m) V^\pi (j) \right) \\
		&\le \sum_{a \in  {\mathcal{A}}} \pi_{ia} \left( \bar{r}_{ia}(m) + \sum_{j \in S} \alpha \bar{P}_{ija}(m) V^\ast (j) \right) \\
		&< \sum_{a \in  {\mathcal{A}}} \pi_{ia} \max_{a' \in A_i} \left\{ \bar{r}_{ia'}(m) + \sum_{j \in S} \alpha \bar{P}_{ija'}(m) V^\ast(j) \right\} \\
		&= \max_{a' \in  {\mathcal{A}}} \left\{ \bar{r}_{ia'}(m) + \sum_{j \in S} \alpha \bar{P}_{ija'}(m) V^\ast(j) \right\} = V^\ast(i).
		\end{align*} Note that the second lines follows from $V^\pi \le V^\ast$ and the third line follows from the fact that $\tilde{a}\notin O_i$ is chosen with positive probability $\pi_{i \tilde{a}}$.
		
		As now $V^\pi(i)< V^\ast(i)$ and by the finiteness of $\mathcal{S}$ and $\mathcal{A}$ an optimal strategy achieving value $V^\ast$ exists, it follows that $\pi$ is not optimal.  
	\end{proof} 
	
	We note that we reduced the problem of determining which of the infinitely many strategies are indeed optimal for a given $m \in \mathcal{P}(\mathcal{S})$ to the problem of determining which of the finitely many deterministic strategies are optimal.
	
	This theorem yields a basic guideline for finding all mean field equilibria:
	{For each point $m \in \mathcal{P}(\mathcal{S})$ determine the set $\mathcal{D}(m)$ of all optimal strategies. 
		Since there are only finitely many deterministic stationary strategies, this yields to a partition of $\mathcal{P}(\mathcal{S})$. 
		Let us write $\text{Opt}(A_1 \times \ldots \times A_S)$ for the set of all $m \in \mathcal{P}(\mathcal{S})$ such that $d \in \mathcal{D}(m)$ if and only if $d_i \in A_i$ for all $i \in \mathcal{S}$. 
		Thus, we then have to search in each of the sets $\text{Opt}(A_1 \times \ldots \times A_S)$ for fixed points of the dynamics given all those deterministic stationary strategies satisfying $d(i) \in A_i$ for all $i\in \mathcal{S}$.}
	
	Furthermore, the results allows us to prove 
	{that a game that is not trivial in the sense that is there are two different population distributions such that different actions are optimal for each of them, has a closed, non-empty set of points where infinitely many strategies are optimal.
		Thus, we indeed have to consider infinitely many (potentially different) fixed point problems in order to compute all stationary mean field equilibria.}
	{Indeed,} by the classical theory of Markov decision processes we know that those deterministic strategies are optimal that maximize the expected discounted reward $V^d(m)$. Noting that the expected discounted reward is given by $$V^d(m) = (\beta I - Q^d(m))^{-1} r^d(m)$$ and that $r^d(\cdot)$ and $Q^d(\cdot)$ are continuous, it follows that also $V^d(\cdot)$ is continuous. The set of all those points where a certain strategy $d$ is optimal is the preimage of $[0,\infty)^S$ under the continuous map $V^d(m) - \max_{d' \in D^s} V^{d'}(m)$ and thus closed. Whenever there are two strategies that are optimal for distinct population distributions we have two (or more) non-empty closed sets  {of points $m \in \mathcal{P}(\mathcal{S})$ for which a certain deterministic strategy is one (but possibly not the only) optimal strategy}. Since $\mathcal{P}(\mathcal{S})$ is itself closed and connected, there will be a non-empty, closed set for which at least two deterministic stationary strategies, and thus infinitely many stationary strategies are optimal.
	
	\section{Existence}
	\label{Section:Existence}
	
	{In Section \ref{Section:IndivControl} we proved that there is a stationary strategy that is optimal among all (also time-dependent) strategies. Moreover, we proved that a stationary strategy is optimal for $m \in \mathcal{P}(\mathcal{S})$ if and only if it is a convex combination from $\mathcal{D}(m)$, which is the set of all deterministic stationary strategies that are optimal.}

	Using this, we will prove that whenever the assumption \ref{assumption:continuous} holds there exists a stationary mean field equilibrium. 
	We will adapt the ideas presented in \citet{DoncelExplicit2016} to prove this. More precisely, we will show the existence of a fixed point of an associated best response map in the dynamics. This map maps to each point $m$ all the stationary points of $Q^\pi(m)$ given that $\pi$ is an optimal strategy for $m$. In contrast to the proof of the existence of dynamic equilibria presented in \citet{DoncelExplicit2016} we do not only rely on standard calculus arguments, but instead the proof crucially relies on our probabilistic representation of the problem.
	
	{We define the best response map $\phi: \mathcal{P}(\mathcal{S}) \rightarrow 2^{\mathcal{P}(\mathcal{S})}$, where $2^{\mathcal{P}(\mathcal{S})}$ denotes the power set of $\mathcal{P}(\mathcal{S})$, by setting 
		$$\phi(m) := \{x\in \mathcal{P}(\mathcal{S}) | \exists \pi \in \text{conv}(\mathcal{D}(m)) : 0 = x^T Q^\pi(m) \}.$$} 
	
	We will now show using Kakutani's fixed point theorem that this map has a fixed point and that each fixed point of this map induces a stationary mean field equilibrium:

	\begin{theorem}
		Given assumption \ref{assumption:continuous} there is a stationary mean field equilibrium.
	\end{theorem}

	\begin{proof}
		We show that $\phi$ has a fixed point using Kakutani's fixed point theorem \citep[Corollary 15.3]{BorderFP1985}, since any such fixed point yields to a stationary mean field equilibrium: Indeed, for any fixed point $m$ we find a strategy $\pi\in \mathcal{D}(m)$ such that $0= m^TQ^\pi(m)$. Since by Lemma \ref{preliminaries:CTMDPEquivalence} and Theorem \ref{preliminaries:OptStrategiesConvexConti} we moreover have that $V_{x_0}(\pi,m) \ge V_{x_0}(\pi',m)$ for all $x_0 \in \mathcal{P}(\mathcal{S})$ and all $\pi' \in \Pi$ the pair $(m,\pi)$ constitutes a stationary mean field equilibrium.
		
		We first note that $\phi(m)$ is \textit{non-empty} for all $m \in \mathcal{P}(\mathcal{S})$. By Theorem \ref{preliminaries:OptStrategiesConvexConti} the set $\mathcal{D}(m)$ is non-empty. Since any continuous time Markov chain with finite state space has at least one stationary distribution there exists an $x \in \mathcal{P}(\mathcal{S})$ such that $0=x^T Q^\pi(m)$, which yields that $x \in \phi(m)$.
		
		Furthermore, for each $m \in \mathcal{P}(\mathcal{S})$ the set $\phi(m)$ is \textit{convex}: Let $x^1, x^2 \in \phi(m)$ be two distinct points. Then, by definition of $\phi$, we find two strategies $\pi^1, \pi^2 \in \text{conv}(\mathcal{D}(m))$ such that $$ 0 = \sum_{i \in \mathcal{S}} \sum_{a \in \mathcal{A}} x_i^1 Q_{ija}(m) \pi^1_{ia} \quad \text{and} \quad 0 = \sum_{i \in \mathcal{S}} \sum_{a \in \mathcal{A}} x_i^2 Q_{ija}(m) \pi^2_{ia}.$$ Define $z_{ia}^1:= x_i^1 \pi_{ia}^1$ and $z_{ia}^2= x_i^2 \pi_{ia}^2$, which satisfy $$ 0 = \sum_{i \in \mathcal{S}} \sum_{a \in \mathcal{A}}  Q_{ija}(m) z^1_{ia} \quad \text{and} \quad 0 = \sum_{i \in \mathcal{S}} \sum_{a \in \mathcal{A}} Q_{ija}(m) z^2_{ia}.$$ Now let $\alpha \in [0,1]$ be arbitrary and define $x^3 = \alpha x^1 + (1-\alpha) x^2$. Then $z^3 = \alpha z^1 + (1-\alpha) z^2$ satisfies $$0 = \sum_{i \in \mathcal{S}} \sum_{a \in \mathcal{A}}  Q_{ija}(m) z^3_{ia},$$ which means that 
		$$\pi_{ia}^3 := \begin{cases}
		0 & \text{if } x_{i}^3=0 \\
		z_{ia}^3/x_{i}^3 &\text{if } x_i^3 >0	
		\end{cases}$$ 
		satisfies $0 = (x^3)^T Q^{\pi^3}(m)$. It remains to verify that $\pi^3 \in \text{conv} (\mathcal{D}(m))$. For this we note that $\pi_{ia}^3>0$ of and only if $z_{ia}^3>0$, which in turn is equivalent to the requirement that $z_{ia}^1>0$ or $z_{ia}^2>0$. This can only happen if $\pi_{ia}^1>0$ or $\pi_{ia}^2>0$. Thus, since $\pi^1,\pi^2 \in \text{conv}(\mathcal{D}(m))$, also $\pi^3 \in \text{conv}(\mathcal{D}(m))$.
		
		We now verify that $\phi$ has a \textit{closed graph}, that is, that for any sequence $(m^n)_{n \in \mathbb{N}} \in  \mathcal{P}(\mathcal{S})^\mathbb{N}$ and $x^n \in \phi(m^n)$ for all $n \in \mathbb{N}$ with $\lim_{n \rightarrow \infty} m^n = m$ and $\lim_{n \rightarrow \infty} x^n = x$ we indeed have $x \in \phi(m)$:	
		Let $(m^n,x^n)_{n \in \mathbb{N}}$ be a converging sequence satisfying $x^n \in \phi(m^n)$ for all $n \in \mathbb{N}$. We denote its limit by $(m,x)$. 
		By definition of $\phi$ we find a sequence $\pi^n \in \text{conv}(\mathcal{D}(m^n))$ such that $0=x^nQ^{\pi^n}(m^n)$. 
		By compactness of $\Pi^s$ we find a converging subsequence $(\pi^{n_k})_{k \in \mathbb{N}}$ with limit $\pi\in \Pi^s$.	
		For any $k \in \mathbb{N}$ let $A_1^k \times \ldots \times A_S^k \subseteq \mathcal{A}^S$ be such that $\pi_{ia}^{n_k}>0$ for all $i \in \mathcal{S}, a \in A_i^k$ and $\pi_{ia}=0$ for all $i \in \mathcal{S}, a \notin A_i^k$. 
		Since $\mathcal{A}^S$ is finite we find a set $A_1 \times \ldots \times A_S$ that occurs infinitely often in the sequence $(A_1^k \times \ldots \times A_S^k)_{k \in \mathbb{N}}$. 
		From this we obtain that $\pi_{ia}=0$ for all $i \in \mathcal{S}$ and $a \notin A_i$.
		Moreover, since $ \pi^n \in \text{conv} (\mathcal{D}(m^n))$ we obtain that for all $l \in \mathbb{N}$ such that $A_1 \times \ldots \times A_S = A_1^l \times \ldots \times A_s^l$ we have $V^{d} (m^{n_l}) = V^\ast(m^{n_l})$ for all deterministic strategies satisfying $d(i) \in A_i$ for all $i \in \mathcal{S}$. By continuity of $V^d(\cdot)$ and $V^\ast(\cdot)$, we obtain that for all these strategies $V^d(m) = V^\ast(m)$. Thus, $\pi \in \text{conv}(\mathcal{D}(m))$.
		Furthermore, by continuity of $Q(\cdot)$, we obtain that
		$$ 0 = \sum_{i \in \mathcal{A}} \sum_{a \in \mathcal{A}} x_i^{n_k} Q_{ija}(m^{n_k}) \pi_{ia}^{n_k} \leftarrow  \sum_{i \in \mathcal{A}} \sum_{a \in \mathcal{A}} x_i Q_{ija}(m) \pi_{ia},$$ which shows that $x \in \phi(m)$.
		
		Using that $\mathcal{P}(\mathcal{A})$ is a compact metric space and that the graph of $\phi$ is closed, we obtain that the values $\phi(m)$ are \textit{compact}  as the limit of any sequence $x^n \in \phi(m)$ lies in $\phi(m)$.
		
		Now Kakutani's fixed point theorem \citep[Corollary 15.3]{BorderFP1985} yields a fixed point $m \in \phi(m)$, which proves the desired claim.
	\end{proof}

	\section{The Fixed Point Problem}
	\label{Section:FixedPointProblem}
	
	\subsection{The Cut Criterion}
	
	To solve the fixed point problem we have to determine the solutions of the equation $m^TQ^\pi(m)=0$ for all strategies $\pi$ that are optimal for some population distribution and then we have to check whether $\pi$ is indeed optimal for these solutions. In many settings this task is not simple (see Section \ref{Section:ExampleCorruption}). However, often a cut criterion similar to the one used for Markov chains is useful, although it is just a reformulation of the balance equation $m^TQ^\pi(m)=0$ (see  \citet[Lemma 1.4]{KellyCut1979} for a description of the criterion for standard continuous time Markov chains).
	The criterion states that if we partition the state space of the Markov chain into two sets, then the probability flow from one set to the other has to equal the probability flow from this other set to the first.
	
	The particular use of the criterion is that in most models that have been consider so far there has always been a set of states for which the dynamics to and from this set cannot be influenced by the player by choosing a particular strategy. This means that any mean field equilibrium irrespective of the chosen strategy has to satisfy certain equations coming from the cut criterion, which could be obtained from the standard balance equations $Q^\pi(m)m=0$ only by sensible rearrangements. In Section \ref{Section:ExampleCorruption} we will show that the criterion indeed simplifies the search for fixed points.
	
	\begin{theorem}
		\label{stationary:CutCriterion}
		Let $\pi$ be a stationary strategy and let $\mathcal{T} \subseteq \mathcal{S}$. Then any stationary population distribution satisfies
		$$\sum_{j \in \mathcal{T}} \sum_{i \in \mathcal{S} \setminus \mathcal{T}} m_i Q_{ij}^\pi(m) = \sum_{j \in \mathcal{T}} \sum_{i \in \mathcal{S}\setminus \mathcal{T}} m_j Q_{ji}^\pi(m).$$
	\end{theorem}
	
	\begin{proof}
		The stationarity condition reads for all $j \in \mathcal{S}$ $$0 = \sum_{i \in \mathcal{S}} m_i Q_{ij}^\pi(m),$$ furthermore, since $Q^\pi(m)$ is conservative, we have for all $j \in \mathcal{S}$ $$m_j \sum_{i \in \mathcal{S}} Q^\pi_{ji}(m)=0.$$ 
		This yields for all $j \in \mathcal{S}$ $$\sum_{i \in \mathcal{S}} m_i Q_{ij}^\pi(m) = m_j \sum_{i \in \mathcal{S}} Q^\pi_{ji}(m).$$
		Summing this identity over all $j \in \mathcal{T}$  yields $$\sum_{j \in \mathcal{T}} \sum_{i \in \mathcal{S}} m_i Q_{ij}^\pi(m) = \sum_{j \in \mathcal{T}} \sum_{i \in \mathcal{S}} m_j Q_{ji}^\pi(m).$$ 
		Subtracting the identity $$\sum_{j \in \mathcal{T}} \sum_{i \in \mathcal{T}} m_i Q_{ij}^\pi(m) = \sum_{j \in \mathcal{T}} \sum_{i \in \mathcal{T}} m_j Q_{ji}^\pi(m)$$ yields the desired result.
	\end{proof}
	
	\subsection{Mean Field Equilibria are Fixed Points of a Specific Map}
	
	This section is devoted to proving an explicit characterization of $\phi(m)$, which has been introduced in Section \ref{Section:Existence}, in terms of the deterministic maps $x^d(\cdot)$ for those strategies $d$ that are optimal for $m$. In order to show this we will need irreducibility of $Q^\pi(m)$ for all strategies $\pi \in \Pi^s$.  Note that it is sufficient to verify irreducibility for all deterministic strategies $d \in D^s$ since any stationary strategy $\pi$ is a convex combination of deterministic strategies and thus $Q^\pi = \sum_{d \in D ^s} \lambda_d Q^d$ is also irreducible.
	
	The main consequence of $Q^\pi(m)$ being irreducible is that there is a unique stationary distribution of the continuous time Markov chain (CTMC) with generator $Q^\pi(m)$ (see \citet[Corollary I.4.6 and Theorem I.4.7]{DurrettEssentials1999}, \citet[Theorem 3.5.1]{NorrisMarkov1997}). This observation allows us to formulate the main theorem, which we will prove in the rest of the section:
	
	\begin{theorem}
		\label{stationary:CharacterizationOfPhiM}
		Let $m \in \mathcal{P}(\mathcal{S})$ such that $Q^d(m)$ is irreducible for all $d \in D^s$.
		Let, furthermore,  $\mathcal{D}(m) = \{d_1, \ldots, d_n\}$ be the set of all deterministic optimal strategies for $m$. 
		Then $$\phi(m) = \text{conv} (x^{d_1}(m), \ldots, x^{d_n}(m)),$$ with $x=x^{d_k}(m)$ being the unique solution of $0 = \sum_{i \in \mathcal{S}} \sum_{a \in \mathcal{A}} x_i Q_{ija}(m) d_{ia}^k$.
	\end{theorem}
	
	The proof of this theorem basically relies on the idea to characterize the stationary distribution $x^\pi(m)$ of the CTMC with irreducible generator $Q^\pi(m)$ by a closed form expression and to show thereafter that $x^\pi(m)$ is a convex combination of $(x^d(m))_{d \in D^s}$. 
	In order to follow this programme we have to prove several properties of the generator matrix $Q$ starting with the following lemma regarding the structural properties of an irreducible, conservative generator $Q \in \mathbb{R}^{S \times S}$, more precisely regarding the minor $Q_{SS}'$, which arises from $Q$ by deleting the last row and column:
	
	\begin{lemma}
		\label{stationary:Lemma:Streichungsmatrix}
		Let $Q \in \mathbb{R}^{S \times S}$ be an irreducible, conservative generator matrix. 
		Then all eigenvalues of the minor $Q_{SS}'$ have negative real part. 
		Consequently $Q_{SS}'$ has full rank and $$\text{sign} (\det(Q_{SS}'))= (-1)^{S+1}.$$
	\end{lemma}
	
	The technical proof can be found in the appendix.
	
	Noting the last column is the negative sum of all other columns we obtain the following corollary:
	\begin{corollary}
		\label{stationary:RankOfQ}
		Let $Q \in \mathbb{R}^{S\times S}$ be an irreducible, conservative generator matrix. Then the rank of $Q$ is $S-1$.
	\end{corollary}
	
	With these two results, we can now explicitly characterize the stationary distribution given a stationary strategy $\pi$ and a population distribution $m$:
	
	\begin{lemma}
		\label{stationary:ExplicitRepresentation}
		Let $\pi \in \Pi^s$ be a stationary strategy and $m \in \mathcal{P}(\mathcal{S})$ such that $Q^\pi(m)$ is irreducible.
		Let $\tilde{Q}^\pi(m)$ be the transpose $(Q^\pi(m))^T$ where the last row is replaced by $(1, \ldots, 1)$.
		Then we have that the unique stationary distribution $x^\pi(m)$ is given by \begin{equation}
		\label{stationary:ExplicitRepresentationFormula}
		x^\pi(m) = (\tilde{Q}^\pi(m))^{-1} \cdot (0,\ldots, 0,1)^T.
		\end{equation}  
		Furthermore, we can write
		$$x^\pi(m)_i = \frac{1}{\det(\tilde{Q}^\pi(m))} (-1)^{S+i} \det (Q^\pi(m))_{iS}',$$ with $$\det(\tilde{Q}^\pi(m)) = \sum_{k=1}^S (-1)^{S+k} \det((Q^\pi(m))_{kS}'.$$
	\end{lemma}
	
	\begin{proof}
		The stationary distribution of our process is uniquely determined by \citep[Theorem II.4.2]{AsmussenQueues2003}
		\begin{equation}
		\label{Stationary:EquationsForStatDist}
		0 = \sum_{i \in \mathcal{S}} x_i^\pi (m) Q_{ij}^\pi(m) \quad \text{ for all } j \in \mathcal{S} \quad \text{and} \quad \sum_{i \in \mathcal{S}} x^\pi_i(m)=1.
		\end{equation}
		Since the last equation of the system $0=x^\pi(m)^TQ^\pi(m)$ is the negative sum of all other equations, we obtain that the system \eqref{Stationary:EquationsForStatDist} is equivalent to $$\begin{pmatrix}
		0 \\ \vdots \\ 0 \\ 1
		\end{pmatrix} = \begin{pmatrix}
		Q^\pi_{11}(m) & \ldots & Q^\pi_{1,S-1}(m) & 1 \\
		\vdots & & \vdots & \vdots \\
		Q^\pi_{S,1}(m) & & Q_{S,S-1}^\pi(m) & 1
		\end{pmatrix}^T x^\pi(m),$$ which by definition of $\tilde{Q}^\pi$ is $(0, \ldots, 0,1)^T = \tilde{Q}^\pi(m) x^\pi(m)$.
		
		We now show that the rank of the matrix $\tilde{Q}^\pi(m)$ is $S$, as in this case we can invert the matrix.	
		As in \citet[p.137-139]{ResnickAdventures} we rely on the existence of the stationary distribution given $Q^\pi(m)$.	
		In order to show that $\tilde{Q}^\pi(m)$ has full rank, we show that $y^T\tilde{Q}^\pi(m)  =0$ implies that $y=0$. 
		For the stationary distribution $x^\pi(m)$ given $Q^\pi(m)$ we have 
		$$0=\left(y^T\tilde{Q}^\pi(m)\right) x^\pi(m) = y^T \left( \tilde{Q}^\pi(m) x^\pi(m) \right) = y^T (0, \ldots, 1)^T = y_S.$$
		Thus,
		$$0 =y^T\tilde{Q}^\pi(m) = y^T \begin{pmatrix}
		Q_{11}^\pi(m) & \ldots & Q_{S1}^\pi(m) \\
		\vdots & \ldots & \vdots \\
		Q_{1,S-1}^\pi(m) & \ldots & Q_{S,S-1}^\pi(m) \\
		1 & \ldots & 1
		\end{pmatrix}$$ implies that $$0 = (y_1, \ldots, y_{S-1}) (Q^\pi(m))_{SS}'.$$
		Since by Lemma \ref{stationary:Lemma:Streichungsmatrix} the matrix $(Q^\pi(m))_{SS}'$ has full rank we obtain that $y_1=\ldots =y_{S-1}=0$, which proves that $\tilde{Q}^\pi(m)$ has full rank.
		
		The last part of the statement simply follows from Cramer's rule together with the Laplace expansion of $\det(\tilde{Q}^\pi(m))$ along the last line $$\det(\tilde{Q}^\pi(m)) = \sum_{k=1}^S (-1)^{S+k} \cdot 1 \cdot \det((\tilde{Q}^\pi(m))_{Sk}')$$ and the observation that $$\det((\tilde{Q}^\pi(m))_{Sk}') = \det((\det(Q^\pi)^T)'_{Sk}) = \det((Q^\pi)_{kS}'),$$ as $\tilde{Q}^\pi(m)$ differs from $Q^\pi(m)^T$ only in the $S$-th row.
	\end{proof}
	
	In order to establish the desired result on characterizing the convex set $\phi(m)$ solely in terms of transition rates and rewards for deterministic strategies, one final preparation has to be made: We have to show that the determinant of $\tilde{Q}^\pi(m)$ has uniform sign over all $\pi \in \Pi^s$: For this write $d^{(a_1, \ldots, a_S)}$ for the deterministic strategy satisfying $d(i) =a_i$. Then it holds that 
	$$\tilde{Q}^\pi(m) := \sum_{(a_1, \ldots, a_S) \in \mathcal{A}^S} \left( \prod_{i=1}^{S} \pi_{ia_i} \right) \tilde{Q}^{d^{(a_1, \ldots,a_S)}}(m).$$ Now a simple application of the intermediate value theorem yields the desired result:
	
	\begin{lemma}
		\label{stationary:Lemma:SignofDeterminant}
		Let $m \in \mathcal{P}(\mathcal{S})$ be a population distribution such that $Q^d(m)$ is irreducible for all $d \in D^s$. Then $\det(\tilde{Q}^\pi(m))$ has uniform sign over $\pi \in \Pi^s$.
	\end{lemma}
	
	\begin{proof}
		We note that the map $\pi \mapsto \tilde{Q}^\pi(m)$, which ranges from $\Pi^s$ to $\mathbb{R}^{S \times S}$ is continuous. Since the determinant is also a continuous function, we see that $\pi \mapsto \det (\tilde{Q}^\pi(m))$ is a continuous function. By Lemma \ref{stationary:ExplicitRepresentation} we have that $\det(\tilde{Q}^\pi(m)) \neq 0$ for all $\pi \in \Pi^s$. If there would be a strategy $\pi_1$ and a strategy $\pi_2$ such that $\det (\tilde{Q}^{\pi_1}(m))<0$ and $\det (\tilde{Q}^{\pi_2}(m)) >0$, then, by the intermediate value theorem, there would be a $\pi \in \Pi^s$ such that $\det (\tilde{Q}^\pi(m))=0$, which would be a contradiction.
	\end{proof}
	
	With all these preparations, we can now prove the characterization result stated in the beginning of the section:
	
	\begin{proof}[Proof of Theorem \ref{stationary:CharacterizationOfPhiM}]
		For readability we suppress the dependence of $Q$ and $x$ on $m$.
		
		Let $\pi = \sum_{(a_1, \ldots, a_n) \in \mathcal{A}^S} \lambda_{(a_1, \ldots, a_S)} d^{(a_1, \ldots, a_S)}$.
		Note that $\lambda_{(a_1, \ldots, a_S)}$ is zero for all non-optimal strategies by Theorem \ref{preliminaries:OptStrategiesConvexConti}.
		With $\hat{Q}$ being the matrix $Q^T$ without the last row, we can now write $\tilde{Q}^\pi$ as follows:
		$$\tilde{Q}^\pi = \begin{pmatrix}
		\sum_{a_1 \in \mathcal{A}} \pi_{1a_1} \hat{Q}_{1,\cdot ,a_1} & \ldots & \sum_{a_S \in \mathcal{S}} \pi_{Sa_S} \hat{Q}_{S,\cdot, a_S} \\
		1 & \ldots & 1 \\
		\end{pmatrix}.$$
		As the determinant is linear in columns and we have $\sum_{a_i \in \mathcal{A}} \pi_{ia_i} = 1$ for all $i \in \mathcal{S}$ we obtain
		\begin{align}
		\label{stationary:equation:determinantTildeQ}
		\det (\tilde{Q}^\pi) 
		&= \sum_{a_1 \in \mathcal{A}} \pi_{1a_1} \det \begin{pmatrix}
		\hat{Q}_{1,\cdot, a_1} & \sum_{a_2 \in \mathcal{A}} \pi_{2a_2} \hat{Q}_{2,\cdot, a_2} & \ldots &  \sum_{a_S \in \mathcal{S}} \pi_{Sa_S} \hat{Q}_{S,\cdot ,a_S} \\ \notag
		1 & 1 & \ldots & 1 
		\end{pmatrix} \\ \notag
		&= \ldots  \\ \notag
		&= \sum_{(a_1, \ldots, a_S)\in \mathcal{A}^S} \pi_{1a_1} \ldots \pi_{Sa_S} \det \begin{pmatrix}
		\hat{Q}_{1,\cdot, a_1} &\ldots & \hat{Q}_{S, \cdot, a_S} \\ 1& \ldots & 1 
		\end{pmatrix} \\
		&= \sum_{(a_1, \ldots, a_S)\in \mathcal{A}^S} \pi_{1a_1} \ldots \pi_{Sa_S} \det (\tilde{Q}^{d^{(a_1, \ldots, a_S)}})
		\end{align} 
		Similarly, we obtain using that  $$\left( Q^{d^{(a_1, \ldots, a_{i-1},\tilde{a}_i, a_{i+1}, \ldots, a_S )}} \right)_{Si}' = \left( Q^{d^{(a_1, \ldots, a_S)}}\right)_{Si}'$$ holds for all $a_1, \ldots, a_S, \tilde{a}_i \in \mathcal{A}$, that
		\begin{align*}
		&\det (\tilde{Q}^\pi)_{Si}' \\
		&= \sum_{(a_1, \ldots, a_{i-1},a_{i+1}, \ldots a_S) \in \mathcal{A}^{S-1}} \pi_{1a_1} \cdot \ldots  {\cdot \pi_{i-1,a_{i-1}} \cdot \pi_{i+1,a_{i+1}} \cdot \ldots} \cdot \pi_{Sa_S} \det (\tilde{Q}^{d^{(a_1, \ldots, a_{i-1}, \tilde{a}_i, a_{i+1}, \ldots, a_S)}})_{Si}'
		\end{align*} for all $\tilde{a}_{i} \in \mathcal{A}$.
		This implies
		\begin{align*}
		(x^\pi)_i 
		&= \frac{1}{\det \tilde{Q}^\pi} (-1)^{S+i} \det ((\tilde{Q}^\pi)_{Si}') \\
		&= \frac{1}{\det \tilde{Q}^\pi} (-1)^{S+i}  \sum_{a_i \in \mathcal{A}} \pi_{ia_i} \det ((\tilde{Q}^\pi)_{Si}') \\
		&= \frac{1}{\det \tilde{Q}^\pi} (-1)^{S+i}  \sum_{(a_1, \ldots, a_S) \in \mathcal{A}^{S}} \pi_{1a_1} \cdot \ldots \cdot \pi_{Sa_S} \det (\tilde{Q}^{d^{(a_1, \ldots, a_S)}})_{Si}' \\
		&=  \frac{1}{\det \tilde{Q}^\pi} \sum_{(a_1, \ldots,  a_S) \in \mathcal{A}^{S}} \pi_{1a_1} \cdot \ldots \cdot \pi_{Sa_S} (-1)^{S+i} \det (\tilde{Q}^{d^{(a_1, \ldots, a_S)}})_{Si}' \\
		&= \sum_{(a_1, \ldots,  a_S) \in \mathcal{A}^{S}} \frac{\pi_{1a_1} \cdot \ldots \cdot \pi_{Sa_S} \det (\tilde{Q}^{d^{(a_1, \ldots, a_S)}})}{\det(\tilde{Q}^\pi)} \cdot (x^{d^{a_1, \ldots, a_S}})_i.
		\end{align*} Thus, we see that $x^\pi$ is a linear combination of $x^{d^{(a_1, \ldots, a_S)}}$ for any stationary strategy $\pi$ and the coefficients are given by 
		$$\lambda_{(a_1, \ldots, a_S)} = \frac{\pi_{1a_1} \cdot \ldots \cdot \pi_{Sa_S} \det (\tilde{Q}^{d^{a_1, \ldots, a_S}})}{\det(\tilde{Q}^\pi)}.$$
		From \eqref{stationary:equation:determinantTildeQ} we obtain that 
		$$\sum_{(a_1, \ldots, a_S) \in \mathcal{A}^\mathcal{S}} \lambda_{(a_1, \ldots, a_S)} = 1.$$ 
		By Lemma \ref{stationary:Lemma:SignofDeterminant} we note that the signs of the determinants $\tilde{Q}^\pi$ and $\tilde{Q}^{\pi'}$ are the same for any two strategies $\pi, \pi'$. Thus, as $\pi_{ia} \ge 0$ for all $i\in \mathcal{S}, a \in \mathcal{A}$ we obtain that $\lambda_{(a_1, \ldots, a_S)} \ge 0$ for all $(a_1,\ldots, a_S) \in \mathcal{A}^S$. 
		To conclude, writing $\mathcal{D}(m)= \{d^1, \ldots, d^n\}$, every point in $\phi(m)$ is a convex combination of $x^{d_1}, \ldots, x^{d_n}$.  {Moreover, any convex combination of $x^{d^1}, \ldots, x^{d^n}$ is the stationary point given a strategy $\pi \in \text{conv}(\mathcal{D}(m))$, which by Theorem \ref{preliminaries:OptStrategiesConvexConti} yiels that these points lie in $\phi(m)$.}
	\end{proof}
	
	{Thus, in order to find all mean field equilibria it is sufficient to follow the following programme: First, compute for all sets $A_1 \times \ldots, \times A_S \subseteq \mathcal{A}^S$ the set $\text{Opt}(A_1 \times \ldots A_S)$, which collects all those points $m \in \mathcal{P}(\mathcal{S})$ for which $\mathcal{D}(m) = \{d \in D^S: d(i) \in A_i\}$. Second, find all fixed points of the map $(\text{conv}\{x^{(a_1, \ldots, a_S)}(\cdot): (a_1, \ldots, a_S) \in A_1 \times \ldots \times A_S \}$, that lie in $\text{Opt}(A_1 \times \ldots A_S)$. Writing $FP(f)$ for the set of all fixed points of the map $f$, we obtain the following result:}
	
	\begin{theorem}
		\label{Stationary:ParitionResult}
		Assume that there is a set $\mathcal{T} \subseteq \mathcal{P}(\mathcal{S})$ such that for all $m \in \mathcal{T}$ and all $\pi \in \Pi^s$ the matrix $Q^\pi(m)$ is irreducible. 
		Then the set of all distributions lying in $\mathcal{T}$ induced by some stationary mean field equilibrium  is given by 
		\begin{align*}
		\bigcup_{A_1 \times \ldots \times A_S \subseteq \mathcal{A}^S} &\Big(FP (\text{conv}\{x^{(a_1, \ldots, a_S)}(\cdot): (a_1, \ldots, a_S) \in A_1 \times \ldots \times A_S \}  \\
		& \qquad \cap \text{Opt}(A_1 \times \ldots \times A_S) \Big).
		\end{align*}
	\end{theorem}

	In case of constant dynamics, that is $Q_{ija}(m) = Q_{ija}$ for all $m \in \mathcal{P}(\mathcal{S})$ the second step of this programme is simple, since the maps $x^d(\cdot)$ are constant with value $(\tilde{Q}^d)^{-1} \cdot (0,\ldots, 0,1)^T$ and this implies that the unique fixed point is $(\tilde{Q}^d)^{-1} \cdot (0,\ldots, 0,1)^T$. Thus, we can characterize the set of all mean field equilibria as explicitly by only computing the optimality sets and the stationary points given $Q^d$ as follows:
	
	\begin{corollary}
		Let the dynamics be constant, that is $Q_{ija}(m) = Q_{ija}$ for all $m \in \mathcal{P}(\mathcal{S})$, and let the generators $Q^d$ given any strategy $d$ be irreducible.  Then the set of all distributions induced by some stationary MFE is given by 
		\begin{align*}
		\bigcup_{A_1 \times \ldots \times A_S \subseteq \mathcal{A}^S} &\Big( \text{conv}\{(\tilde{Q}^{(a_1, \ldots, a_S)})^{-1} \cdot (0,\ldots, 0,1)^T: (a_1, \ldots, a_S) \in A_1 \times \ldots \times A_S \}  \\
		& \qquad \cap \text{Opt}(A_1 \times \ldots \times A_S) \Big).
		\end{align*}
	\end{corollary}

	\section{Examples}
	\label{Section:Examples}
	
	This section presents two mean field game models, which have been considered in similar versions in the literature before, and illustrates the application of the techniques presented before. 
	{For both examples we will first solve the individual control problem given a fixed population distribution and compute given this the non-empty optimality sets $\text{Opt}(A_1 \times \ldots A_S)$. Thereafter we will solve the fixed point problems relyin on the approaches introduced before and sketch how the full characterization of the equilibria can be obtained.}

	\subsection{A Consumer Choice Model}
	
	The model we present now is a model with constant dynamics, that is, the dynamics do not depend on the current population distribution. The context and utility function are similar to a model introduced by  \citet{GomesSocio2014} as a toy example on which the authors demonstrated numerical methods for a specific class of finite state mean field game models, namely those yielding systems of hyperbolic partial differential equations. Note however, that the action spaces and choice options of the players differ systematically.
	
	The model is inspired by consumer choices in the mobile phone sector. 
	The utility of using a certain provider is increasing in the share of customers using it, whereas the costs are constant. 
	We assume that the utility coming from the other customers in service is given by the isoelastic utility function $\ln(m_i)+s_i$, with $s_i \in \mathbb{R}$. 
	Note that since $\ln(m_i)$ is always negative for our choice of $m_i$, one cannot interpret $s_i$ as costs directly, but one rather has to think of $s_i$ consisting of two components: the costs themselves and some base utility from service provision.
	The players can now choose in our model whether to stick to their provider $i$ or whether to switch to the other provider, in this case the player additionally faces a time-unit switching cost $c_i \ge 0$. 
	Note that for technical reasons it is important that the player always faces  {independent of the chosen strategy} a small risk of going to the other provider.
	
	These choice options differ substantially from the model in \citet{GomesSocio2014}, where the players could continuously control the rates at which they switch to the other state and were facing costs corresponding to the square of the rate. 
	From an applied point of view it is questionable, in particular when agents are not experts in the game at hand, that players indeed understand what it means to control the transition rates of a Markov chain.
	Indeed, economic experiments show that most people cannot understand the true effect of random devices even in simple settings \citep{PalgraveMixedStrategies}.

	The formal description of the model is now given as follows: 
	{For technical reasons (i.e. to define the rewards for the case $m_i=0$ properly) we introduce for small enough $\delta>0$ the function $f_\delta: \mathbb{R} \rightarrow \mathbb{R}$ given by}
	\begin{displaymath}
	y \mapsto \begin{cases}
	\frac{1}{2\delta} y^2 + \frac{\delta}{2} &\text{if } y \le \delta \\
	y &\text{if } y > \delta
	\end{cases},
	\end{displaymath}  {which is increasing on $[0,1]$.}
	{Using this we then define the transition rates and reward functions by}
	\begin{align*}
	Q^\text{change} &= \begin{pmatrix}
	-b & b \\  b & -b
	\end{pmatrix} & Q^\text{stay} &= \begin{pmatrix}
	-\epsilon & \epsilon \\ \epsilon & -\epsilon
	\end{pmatrix} \\
	r^\text{change} &= \begin{pmatrix}
	\ln( {f_\delta(}m_1 {)}) +s_1 -c \\
	\ln( {f_\delta(}m_2 {)}) + s_2 - c
	\end{pmatrix} &
	r^\text{stay} &= \begin{pmatrix}
	\ln( {f_\delta(}m_1 {)}) +s_1  \\
	\ln( {f_\delta(}m_2 {)}) + s_2 
	\end{pmatrix},
	\end{align*} where $0 < \epsilon <b$ and $s_1,s_2,c>0$. 
	
	In order to analyse the model we first solve the optimality equations
	\begin{align*}
	\beta V_1(m) &= \max \{ \ln( {f_\delta(}m_1 {)}) +s_1 - c - bV_1(m) + bV_2(m), \\
	& \quad \ln( {f_\delta(}m_1 {)}) +s_1 - \epsilon V_1(m) + \epsilon V_2(m)  \} \\
	\beta V_2(m) &= \max \{ \ln( {f_\delta(}m_2 {)}) +s_2 - c + b V_1(m) - bV_2(m), \\
	&\quad   \ln( {f_\delta(}m_2 {)}) + s_2 + \epsilon V_1(m) - \epsilon V_2(m)\},
	\end{align*} which directly yield that it is optimal to change in state $1$ if and only if $V_1(m) -V_2(m) \le -\frac{c}{b-\epsilon}$ and that it is optimal to change in state $2$ if and only if $V_1(m) -V_2(m) \ge \frac{c}{b-\epsilon}$. Thus, we know that choosing to change in both states is never optimal. Therefore we focus on the three potentially optimal strategies $\{\text{change}\} \times \{\text{stay}\}$, $\{\text{stay}\} \times \{\text{stay}\}$ and $\{\text{stay}\}\times \{\text{change}\}$. The expected discounted reward given these strategies is 
	\allowdisplaybreaks
	\begin{align*}
	&V^{\{\text{change}\}\times \{\text{stay}\}}(m) \\
	&= \left( \beta I - Q^{\{\text{change}\}\times \{\text{stay}\}} \right)^{-1}  r^{\{\text{change}\}\times \{\text{stay}\}} \\
	&= \begin{pmatrix}
	\beta +b & -b \\ - \epsilon & \beta + \epsilon
	\end{pmatrix}^{-1} \begin{pmatrix}
	\ln( {f_\delta(}m_1 {)}) + s_1 - c \\  \ln( {f_\delta(}m_2 {)}) + s_2
	\end{pmatrix} \\
	&= \frac{1}{\beta (\beta + b+ \epsilon)} \begin{pmatrix}
	\beta + \epsilon & b \\ \epsilon &  \beta +b
	\end{pmatrix} \begin{pmatrix}
	\ln( {f_\delta(}m_1 {)}) + s_1 - c \\  \ln( {f_\delta(}m_2 {)}) + s_2
	\end{pmatrix} \\
	&= \frac{1}{\beta (\beta + b+ \epsilon)} \begin{pmatrix} 
	(\ln( {f_\delta(}m_1 {)})+s_1) \cdot (\beta +\epsilon) - c \cdot (\beta + \epsilon) + (\ln( {f_\delta(}m_2 {)})+s_2) \cdot b \\
	(\ln( {f_\delta(}m_1 {)})+s_1) \cdot \epsilon - c \cdot \epsilon + (\ln( {f_\delta(}m_2 {)})+s_2) \cdot (\beta +b)
	\end{pmatrix} \\
	&V^{\{\text{stay}\} \times \{\text{stay}\}}  (m) \\
	&= \begin{pmatrix}
	\beta + \epsilon & -\epsilon \\ - \epsilon & \beta + \epsilon
	\end{pmatrix}^{-1} \begin{pmatrix}
	\ln( {f_\delta(}m_1 {)}) + s_1  \\ \ln( {f_\delta(}m_2 {)}) + s_2
	\end{pmatrix} \\
	&= \frac{1}{\beta^2 + 2 \beta \epsilon} \begin{pmatrix} (\ln( {f_\delta(}m_1 {)})+s_1) \cdot (\beta + \epsilon) + (\ln( {f_\delta(}m_2 {)})+s_2) \cdot \epsilon \\
	(\ln( {f_\delta(}m_1 {)})+s_1) \cdot \epsilon + (\ln( {f_\delta(}m_2 {)})+s_2) \cdot (\beta + \epsilon)
	\end{pmatrix} \\
	&V^{\{\text{stay}\} \times \{\text{change}\}} (m) \\ 
	&= \begin{pmatrix}
	\beta + \epsilon & -\epsilon \\ - b & \beta + b
	\end{pmatrix}^{-1} \begin{pmatrix}
	\ln( {f_\delta(}m_1 {)}) + s_1  \\  \ln( {f_\delta(}m_2 {)}) + s_2-c
	\end{pmatrix} \\
	&= \frac{1}{\beta(\beta + b+ \epsilon)} \begin{pmatrix}
	(\ln( {f_\delta(}m_1 {)})+s_1) \cdot (\beta + b) - c \cdot \epsilon + ( \ln( {f_\delta(}m_2 {)})+s_2) \cdot \epsilon \\
	(\ln( {f_\delta(}m_1 {)})+s_1) \cdot b - c \cdot (\beta + \epsilon) + (\ln( {f_\delta(}m_2 {)}) + s_2) \cdot (\beta + \epsilon)
	\end{pmatrix}.
	\end{align*} In order to understand which strategy is optimal, we compute the differences $V^d_1(m) -V^d_2(m)$:
	\begin{align*}
	& V_1^{\{\text{change}\} \times \{\text{stay}\}}(m) - V_2^{\{\text{change}\} \times \{\text{stay}\}}(m) \\
	&= \frac{1}{\beta + b+ \epsilon} \left( -c + \ln( {f_\delta(}m_1 {)})-\ln( {f_\delta(}m_2 {)}) + s_1-s_2 \right) \\
	&V_1^{\{\text{stay}\} \times \{\text{stay}\}}(m) - V_2^{\{\text{stay}\} \times \{\text{stay}\}}(m) \\
	&= \frac{1}{\beta + 2 \epsilon} \left( \ln( {f_\delta(}m_1 {)}) -\ln( {f_\delta(}m_2 {)}) + s_1-s_2 \right) \\
	&V_1^{\{\text{stay}\} \times \{\text{change}\}}(m) - V_2^{\{\text{stay}\} \times \{\text{change}\}}(m) \\
	&= \frac{1}{\beta + b+ \epsilon} \left( c + \ln( {f_\delta(}m_1 {)}) - \ln( {f_\delta(}m_2 {)})+s_1-s_2 \right)
	\end{align*}
	
	Using that $f_\delta(\cdot)$ is increasing and that $m_1 = 1-m_2$ we obtain for small enough $\delta>0$ that
	\begin{align*}
	&V_1^{\{\text{change}\} \times \{\text{stay}\}}(m) - V_2^{\{\text{change}\} \times \{\text{stay}\}}(m) \le - \frac{c}{b- \epsilon} \\
	&\Leftrightarrow -c + \ln( {f_\delta(}m_1 {)}) - \ln( {f_\delta(}m_2 {)}) +s_1-s_2 \le  \frac{-c(\beta+b+\epsilon)}{b- \epsilon} \\
	&\Leftrightarrow \ln \left( \frac{m_1}{1-m_1} \right)  \le \frac{-c(\beta + 2 \epsilon)}{b- \epsilon} -s_1+s_2 \\
	&\Leftrightarrow m_1 \le \frac{\exp \left( \frac{-c(\beta + 2\epsilon)}{b-\epsilon}-s_1+s_2 \right)}{1+\exp \left( \frac{-c(\beta + 2\epsilon)}{b-\epsilon}-s_1+s_2 \right)} 
	\end{align*} and analogously we obtain
	\begin{align*}
	&\frac{-c}{b- \epsilon} \le V_1^{\{\text{stay}\} \times \{\text{stay}\}} (m) -V_2^{\{\text{stay}\} \times \{\text{stay}\}} \le  \frac{c}{b-\epsilon} \\
	& \Leftrightarrow \frac{\exp \left( \frac{-c(\beta + 2\epsilon)}{b-\epsilon}-s_1+s_2 \right)}{1+\exp \left( \frac{-c(\beta + 2\epsilon)}{b-\epsilon}-s_1+s_2 \right)}  \le m_1 \le  \frac{\exp \left( \frac{c(\beta + 2\epsilon)}{b-\epsilon}-s_1+s_2 \right)}{1+\exp \left( \frac{c(\beta + 2\epsilon)}{b-\epsilon}-s_1+s_2 \right)} 
	\end{align*} as well as
	\begin{align*}
	&\frac{c}{b-\epsilon} \le V_1^{\{\text{stay}\} \times \{\text{change}\}}(m) - V_2^{\{\text{stay}\} \times \{\text{change}\}}(m) \\
	&\Leftrightarrow  \frac{\exp \left( \frac{c(\beta + 2\epsilon)}{b-\epsilon}-s_1+s_2 \right)}{1+\exp \left( \frac{c(\beta + 2\epsilon)}{b-\epsilon}-s_1+s_2 \right)} \le m_1. 
	\end{align*} As a next step we compute the fixed points given the three strategies $\{\text{change}\} \times \{\text{stay}\}$, $\{\text{stay}\} \times \{\text{stay}\}$ and $\{\text{stay}\}\times \{\text{change}\}$, noting that we never need to consider the fixed point given $\{\text{change}\}\times \{\text{change}\}$, as it is never optimal to randomize in both states. Since we face standard continuous time Markov chains, this task is simple and we obtain for each strategy a unique fixed point:
	$$x^{\{\text{change}\}\times \{\text{stay}\}} = \begin{pmatrix}
	\frac{\epsilon}{b+\epsilon} \\ \frac{b}{b+ \epsilon} 
	\end{pmatrix}, \quad x^{\{\text{stay}\} \times \{\text{stay}\}} = \begin{pmatrix}
	\frac{1}{2} \\ \frac{1}{2}
	\end{pmatrix}, \quad x^{\{\text{stay}\} \times \{\text{change}\}} = \begin{pmatrix}
	\frac{b}{b+\epsilon} \\ \frac{\epsilon}{b+ \epsilon} 
	\end{pmatrix}.$$ Depending on the choice of parameters we have between one and five equilibria: For simplicity we write 
	$$d_1 = \frac{\exp \left( \frac{-c(\beta + 2\epsilon)}{b-\epsilon}-s_1+s_2 \right)}{1+\exp \left( \frac{-c(\beta + 2\epsilon)}{b-\epsilon}-s_1+s_2 \right)} \quad \text{ and } d_2 = \frac{\exp \left( \frac{c(\beta + 2\epsilon)}{b-\epsilon}-s_1+s_2 \right)}{1+\exp \left( \frac{c(\beta + 2\epsilon)}{b-\epsilon}-s_1+s_2 \right)}$$ and note that we always have $d_1 < d_2$.
	
	We can now obtain the exact number and position of equilibria by coefficient comparison. The strategies that are used in mixed strategy equilibria are then obtained by solving the balance equation $m^TQ^\pi(m)=0$ with $m$ being the population distribution of the mixed strategy equilibrium. We will omit this last step, which is basically the task of solving a system of linear equations.
	
	\begin{itemize}
		\item[(i)] If $d_1 < \frac{\epsilon}{b+ \epsilon}$ and $d_2 < \frac{1}{2}$, then $x^{\{\text{stay}\}\times \{\text{change}\}}$ together with the deterministic strategy $\{\text{stay}\} \times \{\text{change}\}$ is the unique stationary mean field equilibrium. 
		
		\item[(ii)] If $d_1 < \frac{\epsilon}{b+ \epsilon}$ and  $\frac{1}{2} \le d_2 \le \frac{b}{b+\epsilon}$, then $x^{\{\text{stay}\}\times \{\text{stay}\}}$ together with the deterministic strategy $\{\text{stay}\} \times \{\text{stay}\}$ and $x^{\{\text{stay}\} \times \{\text{change}\}}$ together with the deterministic strategy $\{\text{stay}\} \times \{\text{change}\}$ are the deterministic equilibria. Furthermore, there is one mixed strategy equilibrium with $m_1 =d_2$ that randomizes over $\{\text{stay}\}$ and $\{\text{change}\}$ in the second state.
		If $k_2 \in \{\frac{1}{2}, \frac{b}{b+\epsilon}\}$, then the mixed strategy equilibrium coincides with the pure strategy equilibrium with the same population distribution.
		
		\item[(iii)] If $d_1< \frac{\epsilon}{b+\epsilon}$ and $\frac{b}{b+\epsilon}<d_2$, then the only equilibrium is given by $x^{\{\text{stay}\} \times \{\text{stay}\}}$ together with the deterministic strategy $\{\text{stay}\} \times \{\text{stay}\}$.
		
		\item[(iv)] If $\frac{\epsilon}{b+ \epsilon}\le d_1 \le \frac{1}{2}$ and $d_2 < \frac{1}{2}$, then  $x^{\{\text{stay}\}\times \{\text{change}\}}$ together with the deterministic strategy $\{\text{stay}\} \times \{\text{change}\}$ and  $x^{\{\text{change}\}\times \{\text{stay}\}}$ together with the deterministic strategy $\{\text{change}\} \times \{\text{stay}\}$ are the deterministic equilibria. Furthermore, there is a mixed strategy equilibrium at the point with $m_1=d_1$ that randomizes over $\{\text{stay}\}$ and $\{\text{change}\}$ in the first state. 
		If $k_1 =\frac{\epsilon}{b+\epsilon}$, then the mixed strategy equilibrium coincides with the pure strategy equilibrium with the same population distribution.
		
		\item[(v)] If $ \frac{\epsilon}{b+ \epsilon} \le d_1 \le \frac{1}{2}$ and $\frac{1}{2} \le d_2 \le \frac{b}{b+\epsilon}$, then $x^{\{\text{stay}\}\times \{\text{change}\}}$ together with the deterministic strategy $\{\text{stay}\} \times \{\text{change}\}$,  $x^{\{\text{stay}\}\times \{\text{stay}\}}$ together with the deterministic strategy $\{\text{stay}\} \times \{\text{stay}\}$ and $x^{\{\text{change}\}\times \{\text{stay}\}}$ together with the deterministic strategy $\{\text{change}\} \times \{\text{stay}\}$ are the deterministic equilibria. Furthermore, there is one mixed strategy equilibrium with $m_1=d_1$ that randomizes over $\{\text{stay}\}$ and $\{\text{change}\}$ in the first state and one mixed strategy equilibrium with $m_1 =d_2$ that randomizes over $\{\text{stay}\}$ and $\{\text{change}\}$ in the second state.
		If $k_1 \in \{\frac{\epsilon}{b+\epsilon}, \frac{1}{2}\}$ or $k_2 \in \{\frac{1}{2}, \frac{b}{b+\epsilon}\}$, then the mixed strategy equilibrium coincides with the pure strategy equilibrium with the same population distribution.
		
		\item[(vi)] If $ \frac{\epsilon}{b+ \epsilon} \le d_1 \le \frac{1}{2}$ and $\frac{b}{b+\epsilon} < d_2$, then $x^{\{\text{stay}\}\times \{\text{stay}\}}$ together with the deterministic strategy $\{\text{stay}\} \times \{\text{stay}\}$ and $x^{\{\text{change}\}\times \{\text{stay}\}}$ together with the deterministic strategy $\{\text{change}\} \times \{\text{stay}\}$ are the deterministic equilibria. Furthermore, there is one mixed strategy equilibrium with $m_1=d_1$ that randomizes over $\{\text{stay}\}$ and $\{\text{change}\}$ in the first state.
		If $k_1 \in \{\frac{\epsilon}{b+\epsilon}, \frac{1}{2}\}$, then the mixed strategy equilibrium coincides with the pure strategy equilibrium with the same population distribution.
		
		\item[(vii)] If $\frac{1}{2}< d_1$ and $\frac{1}{2}< d_2 \le \frac{b}{b+\epsilon}$, then $x^{\{\text{stay}\}\times \{\text{change}\}}$ together with the deterministic strategy $\{\text{stay}\} \times \{\text{change}\}$ and  $x^{\{\text{change}\}\times \{\text{stay}\}}$ together with the deterministic strategy $\{\text{change}\} \times \{\text{stay}\}$ are the deterministic equilibria. Furthermore, there is a mixed strategy equilibrium at the point with $m_1=d_2$ that randomizes over $\{\text{stay}\}$ and $\{\text{change}\}$ in the second state. 
		If $k_2 = \frac{b}{b+\epsilon}$, then the mixed strategy equilibrium coincides with the pure strategy equilibrium with the same population distribution.
		
		\item[(viii)] If $\frac{1}{2}< d_1$ and $\frac{b}{b+\epsilon} < d_2$, then then $x^{\{\text{change}\}\times \{\text{stay}\}}$ together with the deterministic strategy $\{\text{change}\} \times \{\text{stay}\}$ is the unique stationary mean field equilibrium. 
	\end{itemize}

	\subsection{A Simplified Corruption Model}
	\label{Section:ExampleCorruption}

	We now consider a simplified version of the corruption model presented in \citet{KolokoltsovCorruption2017}: 
	In their model a player can be in one of the three states honest ($H$), corrupt ($C$) and reserved ($R$). 
	The corrupt players get a higher wage than the honest players, which in turn get a higher wage than the reserved players that have been convicted to be corrupt. 
	For simplicity we set $w_C = 10$, $w_H =5$ and $w_R = 0$ and exclude the fine for being convicted, which is an additional feature in their model.
	The players can choose given that they are not reserved, whether they want to stay corrupt/honest or whether they want to switch behaviour. 
	In this case they become honest/corrupt with the rate $b$. 
	A player that is reserved is recovered with a fixed rate $r$ and we assume that he will then be honest.
	Additionally the model captures ``social pressure'' in two ways: First, the more players are corrupt the higher is the pressure (one cannot escape) to also become corrupt. Second, the more players are honest the higher is the rate to become convicted to be corrupt.
	In the model of \citet{KolokoltsovCorruption2017} there is also a principal agent that convicts players, for simplicity we decided to ignore this feature of the model as well.

	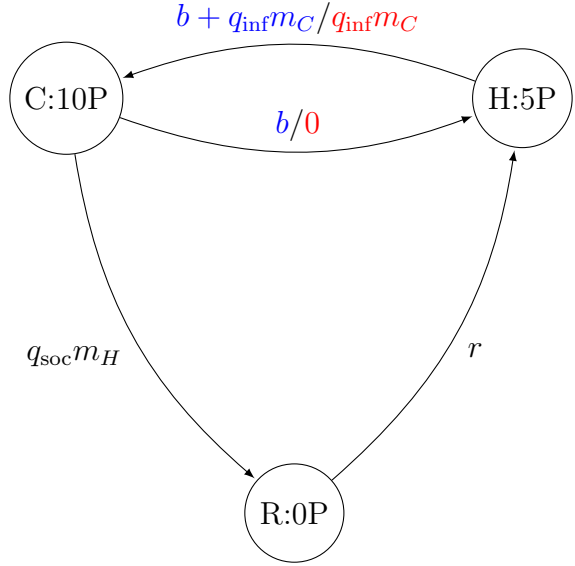
\begin{figure}[h]
		\begin{center}
			\begin{tikzpicture}
			
			\node[state] at (0,0) (A)     {C:10P};
			\node[state] at (6,0) (B)     {H:5P};
			\node[state] at (3, -5.5)  (C)     {R:0P};
			
			\draw[every loop, >=latex, auto=right] 
			(A) edge[bend right=20, auto=left] node {\textcolor{blue}{$b$}/\textcolor{red}{$0$}} (B)
			(B) edge[bend right=20] node {\textcolor{blue}{$b+q_\text{inf}m_C$}/\textcolor{red}{$q_\text{inf}m_C$}} (A)
			(A) edge[bend right=20] node {$q_\text{soc}m_H$} (C)
			(C) edge[bend right=20] node {$r$} (B);
			
			\end{tikzpicture}
		\end{center}
		\caption{Representation of the second example}
		\label{stationary:CorruptionFigure}
	\end{figure}
	
	The formal characterization is given by $\mathcal{A} = \{\text{change},\text{stay}\}$ together with
	\begin{align*}
	Q_\text{change} &= \begin{pmatrix}
	-(b+q_\text{soc}m_H) & b & q_\text{soc}m_H \\
	b+q_\text{inf}m_C & -(b+q_\text{inf}m_C) & 0 \\
	0 & r & -r
	\end{pmatrix} \\
	Q_\text{stay} &= \begin{pmatrix}
	-q_\text{soc}m_H & 0 & q_\text{soc}m_H \\
	q_\text{inf} m_C & -q_\text{inf}m_C & 0 \\
	0 & r & -r
	\end{pmatrix}
	\end{align*}
	and $c_\text{change} = c_\text{stay} = (10,5,0)^T$, where all parameters $b$, $q_\text{inf}$, $q_\text{soc}$ and $r$ are strictly positive. A visualization of this model is given in Figure \ref{stationary:CorruptionFigure}.
	
	We start with computing the value function for given $m \in \mathcal{P}(\mathcal{S})$ as the unique solution of
	\begin{align*}
	\beta V_1(m) &= \max \{ 10 - (b+q_\text{soc}m_H) V_1(m) + bV_2(m) + q_\text{soc}m_H V_3(m), \\
	&\quad  10 -q_\text{soc}m_HV_1(m) +q_\text{soc}m_H V_3(m)\} \\
	\beta V_2(m) &= \max \{ 5 + (b+q_\text{inf} m_C) V_1(m) - (b+q_\text{inf}m_C)V_2, \\
	&5 + q_\text{inf}m_C V_1(m) - q_\text{inf}m_C V_2(m) \} \\
	\beta V_3(m) &= r V_2(m) - rV_3(m).
	\end{align*} 
	
	We directly see that one should choose to change in state $1$ if $V_1(m) \le V_2(m)$ and to stay in state $1$ if $V_1(m) \ge V_2(m)$ and one should chose to change in state $2$ if $V_1(m) \ge V_2(m)$ and to stay in state $2$ if $V_2(m) \ge V_1(m)$. A straightforward calculation yields that the optimality sets are given by
	\begin{align*}
	\text{Opt} (\{\text{stay} \} \times \{\text{change}\}) &= \left\{ m \in \mathcal{P}(\mathcal{S}): m_H < \frac{r+\beta}{q_\text{soc}} \right\} \\
	\text{Opt} (\{\text{change}\} \times \{\text{stay}\}) &= \left\{ m \in \mathcal{P}(\mathcal{S}): m_H > \frac{r+\beta}{q_\text{soc}} \right\} \\
	\text{Opt} (\{\text{stay}, \text{change}\} \times \{\text{stay}, \text{change} \}) &= \left\{m \in \mathcal{P}(\mathcal{S}): m_H = \frac{r+\beta}{q_\text{soc}} \right\}. 
	\end{align*}	
	
	Note however, that depending on the choice of parameters the quantity $\frac{r+\beta}{q_\text{soc}}$ is greater than one, exactly one or less than one.
	In the first case  only $\text{Opt} (\{\text{stay} \} \times \{\text{change}\})$ is non-empty, in the second case $\text{Opt} (\{\text{stay} \} \times \{\text{change}\})$ and $\text{Opt} (\{\text{stay}, \text{change}\} \times \{\text{stay}, \text{change} \})$ are non-empty and in the third case all optimality sets are non-empty.
	
	As a second step we use the cut criterion from Theorem \ref{stationary:CutCriterion} for the set $\mathcal{T} = \{R\}$, which yields the equation $q_\text{soc}m_Hm_C = rm_R$ for all stationary strategies. Together with the equation $1=m_R + m_H + m_C$ we obtain 
	\begin{align*}
	q_\text{soc} m_H m_C = r(1-m_H-m_C) &\Leftrightarrow m_H (q_\text{soc} m_C +r) =  r - rm_C \\
	&\Leftrightarrow m_H = \frac{r(1-m_C) }{q_\text{soc}m_C + r}
	\end{align*} or equivalently $m_C = r(1-m_H)/(q_\text{soc} m_H+r),$ note that this representations directly imply that $m_C \in [0,1] \Leftrightarrow m_H \in [0,1]$. Furthermore we obtain that the sum of $m_C$ and $m_H$ is always less than one, thus any mean field equilibrium is uniquely characterized by describing $m_C$. More precisely, any stationary mean field equilibrium has a distribution of the form 
	\begin{equation}
	\label{stationary:Corruption:All MFESatisfy}
	\left( m_C, \frac{r-rm_C}{q_\text{soc} m_C +r}, \frac{q_\text{soc} - q_\text{soc} m_C^2}{q_\text{soc} m_C + r} \right).
	\end{equation}
	
	It now remains to consider the fixed point problems given the possible optimal strategies: We start by noting that stationary points of the dynamics given $\{\text{stay} \} \times \{\text{change} \}$ exists whenever \eqref{stationary:Corruption:All MFESatisfy} and $$ -q_\text{soc}m_C m_H + q_\text{inf}m_Cm_H + bm_H = 0 \Leftrightarrow m_H(m_C(q_\text{soc}-q_\text{inf})-b)=0$$ is satisfied, which is true if $m_H = 0$ or $m_C = b/(q_\text{soc}-q_\text{inf})$. Thus whenever these points lie in $\text{Opt}(\{\text{stay}\} \times \{\text{change}\})$ or $\text{Opt}(\{\text{stay}, \text{change} \} \times \{\text{stay},\text{change}\})$ we have a deterministic mean field equilibrium given the strategy $\{\text{stay}\} \times \{\text{change}\}$.
	
	Similarly, stationary points of the dynamics given $\{\text{change}\} \times \{\text{stay}\}$ have to satisfy \eqref{stationary:Corruption:All MFESatisfy} and $$-bm_C - q_\text{soc}m_Hm_C + q_\text{inf}m_Cm_H=0 \Leftrightarrow m_C (m_H(q_\text{inf}-q_\text{soc})-b)=0,$$ which is true if either $m_C=0$ or $m_H = b/(q_\text{inf}-q_\text{soc})$. Thus whenever these points lie in $\text{Opt}(\{\text{change}\} \times \{\text{stay}\})$ or $\text{Opt}(\{\text{stay}, \text{change} \} \times \{\text{stay},\text{change}\})$ we have a deterministic mean field equilibrium given the strategy $\{\text{change}\} \times \{\text{stay}\}$.
	
	When searching for stationary points given the dynamics of mixed strategies, which might be equilibria,we can restrict to those that lie inside the set $\text{Opt}(\{\text{stay}, \text{change}\} \times \{\text{stay},\text{change}\})$. In this set all equilibria have to satisfy $m_H = \frac{r+\beta}{q_\text{soc}}$, that is
	\begin{align*}
	\frac{r+\beta}{q_\text{soc}} = \frac{1-m_C}{\frac{q_\text{soc}}{r} m_C +1} 
	&\Leftrightarrow m_C = \frac{r(q_\text{soc}-r-\beta)}{(2 r + \beta) q_\text{soc}}.
	\end{align*} 
	It remains to check whether there is a strategy such that the point 
	$$\left(\frac{r(q_\text{soc}-r-\beta)}{(2 r + \beta) q_\text{soc}}, \frac{r+\beta}{q_\text{soc}},  \frac{(r+ \beta)(q_\text{soc}-r-\beta)}{(2r+ \beta)q_\text{soc}} \right)$$ 
	is indeed a fixed point for the individual dynamics equation given strategy $\pi$, which means that we have to find constants $\pi_{1,\text{change}}$ and $\pi_{2,\text{change}}$ that satisfy for this point 
	$$(- \pi_{1,\text{change}} b - q_\text{soc} m_H)m_C + \pi_{2,\text{change}} b m_H + q_\text{inf} m_Cm_H = 0,$$ which is parameter-dependent.
	
	As in the previous example, we would now need to perform a case analysis to obtain the exact set of mean field equilibria for all possible equilibrium constellations. Additionally, we would need to solve the balance equations $m^TQ(m)=0$ for $\pi$ for any $m$ that is a candidates for randomized equilibria. Both tasks are simple, but tedious and we omit them here.
	
	\appendix
	\section{Appendix}
	\begin{proof}[Proof of Lemma \ref{stationary:Lemma:Streichungsmatrix}]	
		Since we face a conservative generator, we see by definition that $Q_{ij} \ge 0$ for all $i \neq j$ and that $\sum_{j \in \mathcal{S}} Q_{ij} =0$.
		Thus, all off-diagonal entries of $Q$ are non-negative and the row sum is always zero. 
		Furthermore, by requiring irreducibility we do not have a row of zeros, thus the diagonal entries are strictly negative.
		
		The matrix $Q_{SS}'$ again has non-negative off-diagonal entries, strictly negative diagonal entries and row sum is less or equal zero. 
		Furthermore the irreducibility of $Q$ implies that $Q_{iS}$ is non-zero for at least one $i \in \{1, \ldots, S-1\}$ and thus the row sum for at least one $i$ is strictly negative.
		We now show that there exists a vector $x \in \mathbb{R}^{S-1}$ such that $x \ge 0$ and $x_i>0$ for at least one $i \in \{1, \ldots, S-1\}$ such that $-Q_{SS}' x>0$ (where the inequality signs hold pointwise). 
		
		For this we first note that for all $k \in \{1, \ldots, S-1\}$	
		\begin{align*}
		(Q_{SS}' x)_k &= \sum_{l=1}^{S-1} x_l (-Q_{SS}')_{kl} =  \sum_{l=1}^{S-1} x_l (-Q)_{kl} \\
		&= \sum_{l \in \{1, \ldots, S-1\} \setminus \{k\} } -x_l Q_{kl} + x_k \cdot (-Q)_{kk} \\
		&= \sum_{l \in \{1, \ldots, S-1\} \setminus \{k\} } -x_l Q_{kl}  + x_k \sum_{l \in \mathcal{S} \setminus \{k\}} Q_{kl} \\
		&= \sum_{l \in \{1, \ldots, S-1\} \setminus \{k\} } (x_k -x_l)Q_{kl} + x_k Q_{kS}. 
		\end{align*} From which one directly sees that $(Q_{SS}' x)_k$ is increasing in $x_k$ and decreasing in $x_l$ since $Q_{kl}>0$ for all $l \neq k$.
		
		Define for all $x \in ([0,\infty))^{S-1}$ the set $\mathcal{T}(x) = \{ k \in \{1, \ldots, S-1\} : (Q_{SS}'x)_k >0\}$ as the set of all indices where the $i$-th component of  $Q_{SS}'x$ is greater than zero. 
		We will now define inductively an sequence $x^n \in \mathbb{R}^S$ a sequence of vectors such that for some $m \in \mathbb{N}$ we have $\mathcal{T}(x^m)=\{1, \ldots, S-1\}$. 
		We will start with the vector $x^0$ with a one at every component and construct $x^n$ in such a way that $\mathcal{T}(x^{n-1}) \subsetneq \mathcal{T}(x^{n})$ for all $n \in \mathbb{N}$ and moreover $x_i^n>0$ for all $i \in \{1, \ldots, S-1\}$ and all $n \in \mathbb{N}_0$.
		
		Starting with $x^0$ being the vector of ones directly implies that $\mathcal{T}(x^0) \neq \emptyset$, as we have previously seen that the row sum, which is $(Q_{SS}'x^0)_i$ is positive for some index $i \in \{1, \ldots, S-1\}$.
		
		In the $n$-th step ($n \in \mathbb{N}$) we check whether $\mathcal{T}(x^{n-1}) = \{1, \ldots, S-1\}$. 
		If this is the case we have shown that a vector with the desired properties exists, else there is an index $j \in \{ 1, \ldots, S-1\} \setminus \mathcal{T}(x^{n-1})$. 
		In this case let $i \in \mathcal{T}(x^{n-1})$. Since our CTMC is irreducible the underlying transition graph is strongly connected, which implies that there is a path from $j$ to $i$. 
		Let $\tilde{k}$ be the first node on this path that lies in $\mathcal{T}(x^{n-1})$ and let $k$ be its predecessor. Note that by definition of the transition graph $Q_{k\tilde{k}}>0$.
		Now let $x^n$ be as follows $x_i^n = x_i^{n-1}$ for all $i \in \{1, \ldots, S-1\} \setminus \{\tilde{k}\}$ and $x_{\tilde{k}}^n$ such that $$0<x_{\tilde{k}}^n < x_{\tilde{k}}^{n-1} \quad \text{and} \quad (Q_{SS}'x^n)_{\tilde{k}}>0.$$ 
		This is possible as $x_{\tilde{k}}^{n-1}>0$ and $(Q_{SS}'x^{n-1})_{\tilde{k}}>0$ and moreover $(Q_{SS}'x)_{\tilde{k}}$ is continuous and increasing in $ x_{\tilde{k}}$.
		
		It is obvious that $x^n_l>0$ for all $l \in \{1, \ldots, S-1\}$. It remains to check whether $m \in \mathcal{T}(x^{n-1})$ implies $m \in \mathcal{T}(x^n)$ and to show that $k \in \mathcal{T}(x^n)$, as this proves that $\mathcal{T}(x^{n-1}) \subsetneq \mathcal{T}(x^n)$. 
		
		As $Q_{SS}'(x^n)_{\tilde{k}}>0$ by construction it is obvious that $\tilde{k} \in \mathcal{T}(x^n)$.
		For $m \neq \tilde{k}$ we see again as $(Q_{SS}'x)_m$ is decreasing in $x_{\tilde{k}}$ and $x_m^n = x_m^{n-1}$ that 
		\begin{align*}
		(Q_{SS}' x^n)_m &= \sum_{l \in \{1, \ldots, S-1\} \setminus \{m\}} (x_m^n - x_l^n) Q_{ml} + x_m^n Q_{mS}  \\
		&\ge \sum_{l \in \{1, \ldots, S-1\}\setminus \{m\}} (x_m^{n-1} - x_l^{n-1}) Q_{ml} + x_mQ_{mS} = (Q_{SS}'x^{n-1})_m.
		\end{align*} Thus if $m \in \mathcal{T}(x^{n-1})$, then $m \in \mathcal{T}(x^n)$.
		
		For $m=k$ we have a strict inequality as $(x_k^n -x_{\tilde{k}}^n)Q_{k\tilde{k}} >(x_k^{n-1} - x_{\tilde{k}}^{n-1}) Q_{k\tilde{k}}$ since $Q_{k\tilde{k}}>0$ and $x_{\tilde{k}}^n < x_{\tilde{k}}^{n-1}$. 
		Thus $(Q_{SS}' x^n)_k > (Q_{SS}' x^{n-1})_k =0$, which directly implies $k \in \mathcal{T}(x^n)$, thus $\mathcal{T}(x^{n-1})\subsetneq \mathcal{T}(x^{n})$.
		
		Thus, we indeed obtained a vector $x \in \mathbb{R}^{S-1}$ such that $x \ge 0$ and $x_i>0$ for at least one $i \in \{1, \ldots, S-1\}$ such that $-Q_{SS}' x>0$. 
		Now we can conclude that it is a non-singular $M$-matrix in the sense of \citet[Chapter 6]{BermanNonnegative1979} and furthermore that all eigenvalues of our matrix have positive real part. 
		Therefore, all eigenvalues of $Q_{SS}$ have negative real part.
		
		As the matrix $Q_{SS}$ is real, all complex eigenvalues appear in pairs of complex conjugates. 
		As the product of a complex number and its complex conjugate is non-negative, we obtain that the determinant which can be computed as the product of all eigenvalues is as claimed since only the number of real eigenvalues or more specifically the parity of this number matters for the sign of the determinant. 
		As the parity of the number of real eigenvalues always equals the parity of $S-1$, we conclude that the sign pattern of the determinant is as claimed. 
	\end{proof}
	
	\newpage

	\bibliographystyle{plainnat}
	\bibliography{literature}

\begin{thebibliography}{39}
\providecommand{\natexlab}[1]{#1}
\providecommand{\url}[1]{\texttt{#1}}
\expandafter\ifx\csname urlstyle\endcsname\relax
  \providecommand{\doi}[1]{doi: #1}\else
  \providecommand{\doi}{doi: \begingroup \urlstyle{rm}\Url}\fi

\bibitem[Asmussen(2003)]{AsmussenQueues2003}
S{\o}ren Asmussen.
\newblock \emph{Applied Probability and Queues}, volume~51 of \emph{Stochastic
  Modelling and Applied Probability}.
\newblock Springer-Verlag, New York, 2\textsuperscript{nd} edition, 2003.
\newblock ISBN 0-387-00211-1.

\bibitem[Basna et~al.(2014)Basna, Hilbert, and
  Kolokoltsov]{BasnaConvergence2014}
Rani Basna, Astrid Hilbert, and Vassili~N. Kolokoltsov.
\newblock {An epsilon-Nash equilbrium for non-linear Markov games of
  mean-field-type on finite spaces}.
\newblock \emph{Commun. Stoch. Anal.}, 8\penalty0 (4):\penalty0 449--468, 2014.
\newblock \doi{10.31390/cosa.8.4.02}.

\bibitem[Benazzoli et~al.(2018)Benazzoli, Campi, and
  Persio]{BennazzoliMixedStrategy}
Chiara Benazzoli, Luciano Campi, and Luca~Di Persio.
\newblock {Mean field games with controlled jump-diffusion dynamics: Existence
  results and an illiquid interbank market model}, 2018.
\newblock ArXiv preprint arXiv:1703.01919.

\bibitem[Bensoussan et~al.(2013)Bensoussan, Frehse, and Yam]{BensoussanMFG}
Alain Bensoussan, Jens Frehse, and Phillip Yam.
\newblock \emph{Mean Field Games and Mean Field Type Control Theory}.
\newblock SpringerBriefs in Mathematics. Springer, New York, Heidelberg,
  Dordrecht, London, 2013.
\newblock ISBN 978-1-4614-8507-0.

\bibitem[Berman and Plemmons(1979)]{BermanNonnegative1979}
Abraham Berman and Robert~J. Plemmons.
\newblock \emph{Nonnegative matrices in the mathematical sciences}.
\newblock Computer science and applied mathematics. Academic Press, Inc., New
  York, 1979.
\newblock ISBN 0-12-092250-9.

\bibitem[Besancenot and Dogguy(2015)]{BesancenotParadigm2015}
Damien Besancenot and Habib Dogguy.
\newblock {Paradigm Shift: A Mean Field Game Approach}.
\newblock \emph{Bull. Econ. Res.}, 67\penalty0 (3):\penalty0 289--302, 2015.
\newblock \doi{10.1111/boer.12024}.

\bibitem[Border(1985)]{BorderFP1985}
Kim~C. Border.
\newblock \emph{Fixed point theorems with applications to economics and game
  theory}.
\newblock Cambridge University Press, Cambridge, 1985.
\newblock ISBN 0-521-38808-2.

\bibitem[Caines et~al.(2017)Caines, Huang, and Malham\'{e}]{CainesHandbook}
Peter~E. Caines, Minyi Huang, and Roland~P. Malham\'{e}.
\newblock {Mean Field Games}.
\newblock In Tamer Basar and Georges Zaccour, editors, \emph{Handbook of
  Dynamic Game Theory}. Springer, Cham, 2017.
\newblock ISBN 978-3-319-27335-8.
\newblock \doi{10.1007/978-3-319-27335-8_7-1}.

\bibitem[Cardaliaguet et~al.(2013)Cardaliaguet, Lasry, Lions, and
  Porretta]{CardaliaguetContiTrend2}
P.~Cardaliaguet, J.-M. Lasry, P.-L. Lions, and A.~Porretta.
\newblock {Long Time Average of Mean Field Games with a Nonlocal Coupling}.
\newblock \emph{SIAM J. Control Optim.}, 51\penalty0 (5):\penalty0 3558--3591,
  2013.
\newblock \doi{10.1137/120904184}.

\bibitem[Cardaliaguet et~al.(2012)Cardaliaguet, Lasry, Lions, and
  Porretta]{CardaliaguetContiTrend}
Pierre Cardaliaguet, Jean-Michel Lasry, Pierre-Louis Lions, and Alessio
  Porretta.
\newblock Long time average of mean field games.
\newblock \emph{Netw. Heterog. Media}, 7\penalty0 (2):\penalty0 279--301, 2012.
\newblock \doi{10.3934/nhm.2012.7.279}.

\bibitem[Carmona and Delarue(2018{\natexlab{a}})]{CarmonaMFG2018}
Ren\'{e} Carmona and Fran\c{c}ois Delarue.
\newblock \emph{Probabilistic Theory of Mean Field Games with Applications I:
  Mean Field FBSDEs, Control, and Games}, volume~83 of \emph{Probability Theory
  and Stochastic Modelling}.
\newblock Springer International Publishing, 2018{\natexlab{a}}.
\newblock ISBN 978-3-319-58920-6.
\newblock \doi{10.1007/978-3-319-58920-6}.

\bibitem[Carmona and Delarue(2018{\natexlab{b}})]{CarmonaMFGPartTwo2018}
Ren\'{e} Carmona and Fran\c{c}ois Delarue.
\newblock \emph{Probabilistic Theory of Mean Field Games with Applications II:
  Mean Field Games with Common Noise and Master Equations}, volume~84 of
  \emph{Probability Theory and Stochastic Modelling}.
\newblock Springer International Publishing, 2018{\natexlab{b}}.
\newblock ISBN 978-3-319-56436-4.
\newblock \doi{10.1007/978-3-319-56436-4}.

\bibitem[Carmona and Wang(2018)]{CarmonaFinite}
Ren\'{e} Carmona and Peiqi Wang.
\newblock {A Probabilistic Approach to Extended Finite State Mean Field Games},
  2018.
\newblock arXiv preprint arXiv:1808.07635.

\bibitem[Cecchin and Fischer(2018)]{CecchinProbabilistic2018}
Alekos Cecchin and Markus Fischer.
\newblock {Probabilistic Approach to Finite State Mean Field Games}.
\newblock \emph{Appl. Math. Optim.}, 2018.
\newblock \doi{10.1007/s00245-018-9488-7}.

\bibitem[Doncel et~al.(2016{\natexlab{a}})Doncel, Gast, and
  Gaujal]{DoncelExplicit2016}
Josu Doncel, Nicolas Gast, and Bruno Gaujal.
\newblock {Mean-Field Games with Explicit Interactions}.
\newblock working paper or preprint, 2016{\natexlab{a}}.
\newblock URL \url{https://hal.inria.fr/hal-01277098}.

\bibitem[Doncel et~al.(2016{\natexlab{b}})Doncel, Gast, and
  Gaujal]{DoncelShort}
Josu Doncel, Nicolas Gast, and Bruno Gaujal.
\newblock {Are Mean-field Games the Limits of Finite Stochastic Games?}
\newblock \emph{Performance Evaluation Review}, 44\penalty0 (2):\penalty0
  18--20, 2016{\natexlab{b}}.
\newblock \doi{10.1145/3003977.3003984}.

\bibitem[Durrett(1999)]{DurrettEssentials1999}
Rick Durrett.
\newblock \emph{Essentials of Stochastic Processes}.
\newblock Springer Texts in Statistics. Springer-Verlag, New York, 1999.
\newblock ISBN 0-387-98836-X.

\bibitem[Gomes et~al.(2010)Gomes, Mohr, and Souza]{GomesDiscrete2010}
D~A Gomes, J~Mohr, and R~R Souza.
\newblock Discrete time, finite state space mean field games.
\newblock \emph{J. Math. Pures Appl.}, 93\penalty0 (3):\penalty0 308--328,
  2010.
\newblock \doi{10.1016/j.matpur.2009.10.010}.

\bibitem[Gomes et~al.(2013)Gomes, Mohr, and Souza]{GomesConti2013}
D~A Gomes, J~Mohr, and R~R Souza.
\newblock {Continuous Time Finite State Mean Field Games}.
\newblock \emph{Appl. Math. Optim.}, 68\penalty0 (1):\penalty0 99--143, 2013.
\newblock \doi{10.1007/s00245-013-9202-8}.

\bibitem[Gomes et~al.(2014)Gomes, Velho, and Wolfram]{GomesSocio2014}
D~A Gomes, R~M Velho, and M~T Wolfram.
\newblock Socio-economic applications of finite state mean field games.
\newblock \emph{Philos. Trans. R. Soc. Lond., A, Math. Phys. Eng. Sci.},
  372\penalty0 (2028), 2014.
\newblock \doi{10.1098/rsta.2013.0405}.

\bibitem[Gomes et~al.(2015)Gomes, Nurbekyan, and
  Pimentel]{GomesEconomicModelsMFG2015}
D~A Gomes, L~Nurbekyan, and E~A Pimentel.
\newblock \emph{Economic Models and Mean-field Games Theory}.
\newblock 2015.
\newblock ISBN 978-85-244-0404-7.
\newblock URL \url{https://impa.br/wp-content/uploads/2017/04/30CBM_04.pdf}.

\bibitem[Gu\'{e}ant(2009{\natexlab{a}})]{GueantDiss2009}
Olivier Gu\'{e}ant.
\newblock \emph{Mean field games and applications to economics: Secondary
  topic: Discount rates and sustainable development}.
\newblock PhD thesis, Universite Paris Dauphine, 2009{\natexlab{a}}.
\newblock URL \url{www.oliviergueant.com/uploads/4/3/0/9/4309511/these2.pdf}.

\bibitem[Gu\'{e}ant(2009{\natexlab{b}})]{GueantReference2009}
Olivier Gu\'{e}ant.
\newblock A reference case for mean field games models.
\newblock \emph{J. Math. Pures Appl.}, 92\penalty0 (3):\penalty0 276 -- 294,
  2009{\natexlab{b}}.
\newblock \doi{10.1016/j.matpur.2009.04.008}.

\bibitem[Gu\'{e}ant(2011)]{GueantReduction2011}
Olivier Gu\'{e}ant.
\newblock {From infinity to one: The reduction of some mean field games to a
  global control problem}, 2011.
\newblock ArXiv preprint arXiv:1110.3441.

\bibitem[Gu\'{e}ant(2015)]{GueantCongestion2015}
Olivier Gu\'{e}ant.
\newblock {Existence and Uniqueness Results for Mean Field Games with
  Congestion Effect on Graphs}.
\newblock \emph{Appl. Math. Optim.}, 72\penalty0 (2):\penalty0 291--303, 2015.
\newblock \doi{10.1007/s00245-014-9280-2}.

\bibitem[Gu\'{e}ant et~al.(2011)Gu\'{e}ant, Lasry, and
  Lions]{ParisPrinceton2010}
Olivier Gu\'{e}ant, Jean-Michel Lasry, and Pierre-Louis Lions.
\newblock {Mean Field Games and Applications}.
\newblock In \emph{Paris-Princeton Lectures on Mathematical Finance 2010},
  volume 2003 of \emph{Lecture Notes in Mathematics}, pages 205--266.
  Springer-Verlag, Berlin, Heidelberg, 2011.
\newblock ISBN 978-3-642-14660-2.
\newblock \doi{10.1007/978-3-642-14660-2_3}.

\bibitem[Guo and Hern{\'{a}}ndez-Lerma(2009)]{GuoCTMDP2009}
Xianping Guo and On{\'{e}}simo Hern{\'{a}}ndez-Lerma.
\newblock \emph{Continuous-Time Markov Decision Processes: Theory and
  Applications}, volume~62 of \emph{Stochastic Modelling and Applied
  Probability}.
\newblock Springer-Verlag, Berlin, Heidelberg, 2009.
\newblock ISBN 978-3-642-26072-8.

\bibitem[Huang et~al.(2006)Huang, Malham\'{e}, and Caines]{HuangNCE2006}
Minyi Huang, Roland~P. Malham\'{e}, and Peter~E. Caines.
\newblock {Large population stochastic dynamic games: closed-loop McKean-Vlasov
  systems and the Nash certainty equivalence principle}.
\newblock \emph{Commun. Inf. Syst.}, 6\penalty0 (3):\penalty0 221--252, 2006.
\newblock \doi{10.4310/CIS.2006.v6.n3.a5}.

\bibitem[Kakumanu(1977)]{KakumanuUniformization1977}
Prasadarao Kakumanu.
\newblock Relation between continuous and discrete time markovian decision
  problems.
\newblock \emph{Naval Res. Logist. Quart.}, 24\penalty0 (3):\penalty0 431--439,
  1977.
\newblock \doi{10.1002/nav.3800240306}.

\bibitem[Kelly(1979)]{KellyCut1979}
F.~P. Kelly.
\newblock \emph{Reversibility and Stochastic Networks}.
\newblock Wiley series in probability and mathematical statistics. John Wiley
  {\&} Sons Ltd., Chichester, New York, Brisbane, Toronto, 1979.
\newblock ISBN 0-471-27601-4.

\bibitem[Kolokoltsov and Bensoussan(2016)]{KolokoltsovBotnet2016}
V.~N. Kolokoltsov and A.~Bensoussan.
\newblock {Mean-Field-Game Model for Botnet Defense in Cyber-Security}.
\newblock \emph{Appl. Math. Optim.}, 74\penalty0 (3):\penalty0 669--692, 2016.
\newblock \doi{10.1007/s00245-016-9389-6}.

\bibitem[Kolokoltsov and Malafeyev(2017)]{KolokoltsovCorruption2017}
V.N. Kolokoltsov and O.A. Malafeyev.
\newblock {Mean-Field-Game Model of Corruption}.
\newblock \emph{Dyn. Games Appl.}, 7\penalty0 (1):\penalty0 34--47, 2017.
\newblock \doi{10.1007/s13235-015-0175-x}.

\bibitem[Lacker(2015)]{LackerMixedStrategy}
Daniel Lacker.
\newblock {Mean field games via controlled martingale problems: Existence of
  Markovian equilibria}.
\newblock \emph{Stoch. Process. Their Appl.}, 125\penalty0 (7):\penalty0
  2856--2894, 2015.
\newblock \doi{10.1016/j.spa.2015.02.006}.

\bibitem[Lasry and Lions(2007)]{LasryJapanese2007}
Jean-Michel Lasry and Pierre-Louis Lions.
\newblock Mean field games.
\newblock \emph{Jp. J. Math.}, 2\penalty0 (1):\penalty0 229--260, 2007.
\newblock \doi{10.1007/s11537-007-0657-8}.

\bibitem[Norris(1997)]{NorrisMarkov1997}
J.~R. Norris.
\newblock \emph{Markov Chains}.
\newblock Cambridge Series on Statistical and Probabilistic Mathematics.
  Cambridge University Press, Cambridge, New York, Melbourne, Madrid, Cape
  Town, Singapore, S\~{a}o Paulo, 1997.
\newblock ISBN 978-0-521-63396-3.

\bibitem[Puterman(1994)]{PutermanMDP1994}
Martin~L. Puterman.
\newblock \emph{Markov Decision Processes: Discrete Stochastic Dynamic
  Programming}.
\newblock Wiley series in probability and mathematical statistics. John Wiley
  {\&} Sons, Inc., New York, Chichester, Brisbane, Toronto, Singapore, 1994.
\newblock ISBN 0-471-61977-9.

\bibitem[Resnick(1992)]{ResnickAdventures}
Sidney~I. Resnick.
\newblock \emph{Adventures in Stochastic Processes}.
\newblock Birkhäuser, Boston, 1992.
\newblock ISBN 0-8176-3591-2.

\bibitem[Walker and Wooders(2008)]{PalgraveMixedStrategies}
Mark Walker and John Wooders.
\newblock mixed strategy equilibrium.
\newblock In Steven~N. Durlauf and Lawrence~E. Blume, editors, \emph{The New
  Palgrave Dictionary of Economics}, volume~5, pages 628--631. Macmillan
  Publishers Ltd., Hampshire, New York, 2\textsuperscript{nd} edition, 2008.
\newblock ISBN 978-0-230-22641-8.

\bibitem[Walter(1998)]{WalterODE1998}
Wolfgang Walter.
\newblock \emph{Ordinary Differential Equations}, volume 182 of \emph{Graduate
  Texts in Mathematics}.
\newblock Springer-Verlag, New York, 1998.
\newblock ISBN 0-387-98459-3.

\end{thebibliography}

\end{document}